\documentclass[11pt,a4paper]{amsart}%

\usepackage{amsfonts,amsmath,amssymb}
\usepackage{hyperref}
\usepackage{amssymb,amsthm,amsxtra}
\usepackage[usenames]{color}
\usepackage{amscd}
\usepackage{amsthm}
\usepackage{amsfonts}
\usepackage{amssymb}
\usepackage{amsmath}
\usepackage{graphicx}
\usepackage{enumerate}%
\theoremstyle{plain}

\newtheorem*{theorem*}{Theorem}
\newtheorem{lemma}{Lemma}[subsection]
\newtheorem{proposition}[lemma]{Proposition}
\newtheorem{theorem}[lemma]{Theorem}
\newtheorem{corollary}[lemma]{Corollary}

\newtheorem*{conjecture*}{Conjecture}

\newtheorem{thm}[lemma]{Theorem}

\newtheorem{lem}[lemma]{Lemma}
\newtheorem{cor}[lemma]{Corollary}

\newtheorem{introtheorem}{Theorem}

\newtheorem{introcor}[introtheorem]{Corollary}
\newtheorem{introthm}[introtheorem]{Theorem}

\theoremstyle{definition}

\newtheorem{definition}[lemma]{Definition}
\newtheorem{defn}[lemma]{Definition}

\newtheorem{example}[lemma]{Example}

\newtheorem*{example*}{Example}

\theoremstyle{remark}

\newtheorem*{remark*}{Remark}
\newtheorem{rem}[lemma]{Remark}

\newtheorem{remark}[lemma]{Remark}

\theoremstyle{plain}

\newcommand{\tr}{\operatorname{Tr}}

\newcommand{\ind}{\operatorname{ind}}

\newcommand{\ad}{\operatorname{ad}}

\newcommand{\Rep}{\operatorname{Rep}}
\newcommand{\Hom}{\operatorname{Hom}}
\newcommand{\Mat}{\operatorname{Mat}}
\newcommand{\Rad}{\operatorname{Rad}}

\newcommand{\diag}{\operatorname{diag}}

\newcommand{\A}{\mathbb{A}}

\newcommand{\eps}{\varepsilon}

\newcommand{\length}{\operatorname{length}}

\newcommand{\Ind}{\operatorname{Ind}}

\renewcommand{\Im}{\operatorname{Im}}
\newcommand{\Ker}{\operatorname{Ker}}

\newcommand{\Z}{{\mathbb Z}}
\newcommand{\Q}{{\mathbb Q}}
\newcommand{\R}{{\mathbb R}}

\newcommand{\C}{{\mathbb C}}

\newcommand{\proofend}{\hfill$\Box$\smallskip}

\newcommand{\charc}{{\operatorname{char}}}

\newcommand{\Span}{{\operatorname{Span}}}

\newcommand{\Exp}{\operatorname{Exp}}

\newcommand{\Id}{\operatorname{Id}}

\newcommand{\alp}{{\alpha}}

\newcommand{\gam}{{\gamma}}

\newcommand{\Fre}{{Fr\'{e}chet \,}}

\newcommand{\cF}{{\mathcal{F}}}

\newcommand{\g}{{\mathfrak{g}}}
\newcommand{\fa}{{\mathfrak{a}}}

\newcommand{\fg}{{\mathfrak{g}}}

\newcommand{\fn}{{\mathfrak{n}}}
\newcommand{\fk}{{\mathfrak{k}}}
\newcommand{\fu}{{\mathfrak{u}}}
\newcommand{\fv}{{\mathfrak{v}}}
\newcommand{\fw}{{\mathfrak{w}}}

\newcommand{\cO}{{\mathcal{O}}}
\newcommand{\GL}{\operatorname{GL}}

\newcommand{\gl}{{\mathfrak{gl}}}

\newcommand{\Sc}{{\mathcal S}}

\newcommand{\Ad}{\operatorname{Ad}}

\newcommand{\hot}{\widehat{\otimes}}

\newcommand{\cM}{\mathcal{M}}

\newcommand{\cP}{\mathcal{P}}

\newcommand{\fl}{\mathfrak{l}}
\newcommand{\fm}{\mathfrak{m}}
\newcommand{\fq}{\mathfrak{q}}
\newcommand{\fr}{\mathfrak{r}}

\newcommand{\ft}{\mathfrak{t}}

\newcommand{\cV}{\mathcal{V}}
\newcommand{\cW}{\mathcal{W}}

\newcommand{\Dima}[1]{{#1}}
\newcommand{\DimaA}[1]{{#1}}
\newcommand{\DimaB}[1]{{#1}}
\newcommand{\DimaC}[1]{{#1}}
\newcommand{\DimaD}[1]{{{#1}}}
\newcommand{\DimaE}[1]{{{#1}}}
\newcommand{\DimaF}[1]{{{#1}}}
\newcommand{\DimaG}[1]{{{#1}}}

\newcommand{\onto}{{\twoheadrightarrow}}

\newcommand{\into}{{\hookrightarrow}}

\newcommand{\Orb}{\mathcal{O}}
\newcommand{\WF}{\operatorname{WF}}
\newcommand{\WFC}{\operatorname{WFC}}

\renewcommand{\sl}{\mathfrak{sl}}
\newcommand{\F}{\mathbb{F}}

\begin{document}

\author{Raul Gomez}
\address{Department of Mathematics, 593 Malott Hall, Cornell University, Ithaca, NY 14853, USA }
\email{gomez@cornell.edu}

\author{Dmitry Gourevitch}
\address{The Incumbent of Dr. A. Edward Friedmann Career Development Chair in Mathematics, Faculty of Mathematics and Computer Science,
Weizmann Institute of Science,
POB 26, Rehovot 76100, Israel }
\email{dmitry.gourevitch@weizmann.ac.il}


\author{Siddhartha Sahi}
\address{Department of Mathematics, Rutgers University, Hill Center -
Busch Campus, 110 Frelinghuysen Road Piscataway, NJ 08854-8019, USA}
\email{sahi@math.rugers.edu}

\date{\today}
\title{Generalized and degenerate Whittaker models}



\begin{abstract}
We study generalized and degenerate Whittaker models for reductive groups over local fields of characteristic zero (archimedean or non-archimedean). Our main result is the construction of  epimorphisms from the generalized  Whittaker model corresponding to a nilpotent orbit to any degenerate Whittaker model corresponding  to the same orbit, and to certain degenerate Whittaker models corresponding to bigger orbits.
\DimaC{We also give choice-free definitions of generalized and degenerate Whittaker models.
Finally, we explain how our methods imply analogous results for Whittaker-Fourier coefficients of automorphic representations.

For $\GL_n(\F)$ this implies that a smooth admissible representation $\pi$ has a generalized Whittaker model $\cW_{\cO}(\pi)$ corresponding to a nilpotent coadjoint orbit $\cO$ if and only if $\cO$ lies in the (closure of) the wave-front set  $\WF(\pi)$. Previously this was only known to hold for $\F$  non-archimedean  and $\cO$  maximal in $\WF(\pi)$, see \cite{MW}. We also express $\cW_{\cO}(\pi)$ as an iteration of a version of the Bernstein-Zelevinsky derivatives \cite{BZ-Induced,AGS}.
This enables us to extend to   $\GL_n(\R)$ and $\GL_n(\C)$ several further results from \cite{MW} on the dimension of $\cW_{\cO}(\pi)$ and on the exactness of the generalized Whittaker functor.}

%
%

\end{abstract}

\maketitle
\setcounter{tocdepth}{2}
\tableofcontents


%

\section{Introduction and main results}

\subsection{General results}
Let $\F$ be a local field of characteristic zero, $G$  a reductive group defined over $\F$, $\fg$  its Lie algebra and $\fg^*$  the dual space to $\fg$.
A \emph{Whittaker pair} is an ordered pair $(S,\varphi)\in \fg\times \fg^*$ such that $S$ is semi-simple with eigenvalues of the adjoint action $ad(S)$ lying in $\Q$, and $ad^*(S)(\varphi)=-2\varphi$. Note that $\varphi$ is necessarily nilpotent and given by the Killing form pairing with a (unique) nilpotent element $f=f_{\varphi}\in \fg$.
Following \cite{MW} we attach to $(S,\varphi)$ a certain smooth representation $\cW_{S,\varphi}$ of  $G$ called a \emph{degenerate Whittaker model}  for $G$.

Two classes of Whittaker pairs and the corresponding models will be of special interest to us.
If $S$ is a neutral element for $f_\varphi$ (see Definition \ref{def:nuet} below), then we will say that $(S,\varphi)$ is a \emph{neutral pair} and call $\cW_{S,\varphi}$  a \emph{neutral} model or  a \emph{generalized} model (see \cite{Ka85,MW,Ya86,GZ}). The second class consists of Whittaker pairs $(S,\varphi)$ where $S$ is the neutral element of a principal $\sl_2$-triple in $G$; in this case $f_\varphi$ is necessarily a \DimaF{principal}  nilpotent element for a Levi subgroup of $G$, and we will say $(S,\varphi)$ is a {PL pair}, and $\cW_{S,\varphi}$ is a \emph{PL} model or a \emph{principal degenerate} model (see \cite{Zel,MW,BH,GS}).

We will now sketch the definition of $\cW_{S,\varphi}$, referring to \S \ref{subsec:DefMod} below for more details.
Let $\fu\subset \fg$ denote the sum of all eigenspaces of $ad(S)$ with eigenvalues at least 1.
Note that $\fu$ is a nilpotent subalgebra and let $U:=\Exp(\fu)\subset G$ be the corresponding nilpotent subgroup. Fix an additive character of $\F$. Suppose first that 1 is not an eigenvalue of $ad(S)$. Then the restriction of $\varphi$ to $\fu$ is a character of $\fu$, which defines a character $\chi_\varphi$ of $U$. The degenerate Whittaker model is defined to be the Schwartz induction of this character: $\cW_{S,\varphi}:=\ind_{U}^G\chi_{\varphi}$. If 1 is an eigenvalue of $ad(S)$ then consider the anti-symmetric form on $\fu$ given by $\omega_{\varphi}(X,Y):= \varphi([X,Y])$ and let $\fn$ denote the radical of this form. Let  $\fn':=\fn\cap \Ker \varphi$, and let $N':=\Exp(\fn')$. It is easy to show that $N'$ is a normal subgroup of $U$ and $U/N'$ is isomorphic to a (generalized) Heisenberg group, of which $\varphi$ defines a central character $\chi_{\varphi}$. Let $\sigma_{\varphi}$ denote the oscillator representation of $U/N'$ with central character $\chi_{\varphi}$. Consider $\sigma_{\varphi}$ as a representation of $U$ and define $\cW_{S,\varphi}:=\ind_{U}^G\sigma_{\varphi}$.

If  $(S,\varphi)$ is a neutral pair then the generalized Whittaker model  $\cW_{S,\varphi}$
 does not depend on the choice of a neutral $S$, and will thus be denoted $\cW_\varphi$.
\DimaB{Since conjugate nilpotent elements give rise to isomorphic generalized Whittaker models,}
we will also use the notation $\cW_{\cO}$ for a nilpotent coadjoint orbit $\cO$.

We denote by $\overline{G_S\tilde \varphi}\subset \fg^*$ the closure of the orbit of $\tilde \varphi$ under the coadjoint action of the centralizer of $S$ in $G$.

 \begin{introthm}[See \S \ref{sec:PfMain}] \label{thm:main}
Let $(S,\tilde \varphi)$ be a Whittaker pair and let  $\varphi\in \overline{G_S\tilde \varphi}$.  Then there is a $G$-equivariant epimorphism of $\cW_{\varphi}$ onto $\cW_{S,\tilde \varphi}$.
 \end{introthm}


In particular, taking $\tilde \varphi=\varphi$ we see that the generalized  Whittaker model maps onto any degenerate Whittaker model  corresponding to the same $\varphi$. In fact, we prove a more general result (Theorem \ref{thm:main2}) on epimorphisms between pairs of degenerate Whittaker models, which enables one to define a preorder on the pairs $(S,\varphi)$. If we fix $f_\varphi$ to be a regular  nilpotent element for a Levi subgroup of $G$ then the  minimal elements under this preorder are the  \emph{PL Whittaker pairs} (see \S \ref{subsec:PrinDeg}).  The  corresponding  principal degenerate Whittaker models  are inductions of (possibly degenerate) characters of the nilradicals of minimal parabolic subgroups.

\Dima{The above results have applications to the study of Whittaker functionals on representations of $G$. Following \cite{GZ} we briefly recall the necessary background.}

Let $\cM(G)$ denote the category of smooth admissible\footnote{If $\F$ is archimedean then by admissible we mean admissible \Fre representation of moderate growth.} (finitely generated) representations of $G$ (see \cite{BZ,CasGlob,Wal2}).  For $\pi\in \cM(G)$ and a nilpotent orbit $\cO \subset \fg$ denote
\begin{equation}
\cW_{\cO}(\pi):=\Hom_G(\cW_{\cO},\pi^*).
\end{equation}

The study of Whittaker and generalized Whittaker models for
representations of reductive groups over local fields evolved in
connection with the theory of automorphic forms (via their Fourier
coefficients), and has found important applications in both areas.
See for example \cite{Sh74,NPS73,Kos,Ka85,Ya86, WaJI,Ginz,Ji07,GRS_Book}. From the point of view of
representation theory, the space of generalized Whittaker models may
be viewed as one kind of nilpotent invariant associated to smooth
representations.
Another important invariant is the wave front cycle:
\begin{equation}
\WFC (\pi)=\sum_{\scriptstyle \begin{array}{c} \Orb\subset{\g^*}\\ \mbox{nilpotent}\end{array}} c_{\Orb}(\pi)[\Orb],
\end{equation}
defined by Harish-Chandra in the non-archimedean case
and by Howe and Barbasch-Vogan in the archimedean case (\cite{HoWF,BV}; see also
\cite{Ros,SV}). Recently, the behaviour of the wave-front set and the generalized Whittaker models under $\theta$-correspondence was studied in \cite{GZ,LM}.

For $\F$ non-archimedean, M\oe{}glin and Waldspurger \cite{MW} have established that $\WFC(\pi)$ completely controls the spaces of generalized Whittaker models of interest, namely, the set of maximal orbits in $\WFC(\pi)$ coincides with  the set of maximal orbits such that $\cW_{\Orb}(\pi)\neq 0$, and for any orbit $\cO$ in this set we have
\[
c_{\Orb}(\pi)=\dim \cW_{\Orb}(\pi).
\]
In \cite{MW} it is assumed that the residue characteristic is odd. This assumption was recently removed in \cite{Var}. In \cite{Mog}, the main result of \cite{MW} is used in order to prove that \DimaE{for classical groups}, the maximal orbits in the wave-front set of a tempered representation are distinguished\DimaB{, and the maximal orbits in the wave-front set of any admissible representation are special. The latter was recently generalized in \cite{JLS}.} \DimaC{Partial analogs of these results hold also for archimedean $\F$, see \cite{Harris} for the former and \cite{BVClass,BVExc,JosGold} for the latter.}

For  archimedean $\F$, the correspondence  between the wave-front set and non-vanishing of degenerate Whittaker models is not yet (fully) understood, except for several special cases including the representations with the largest Gelfand-Kirillov dimension \cite{Vog,Ma92} and unitary highest weight modules \cite{Ya01}. For the latter, the wave front set was computed earlier in \cite{Pr91a}.
In \cite{Mat},  it is shown in full generality that every orbit $\cO$ with $\cW_{\cO}(\pi)\neq 0$ lies in the  Zariski closure of some orbit in $\WFC(\pi)$.
The paper \cite{GS-Gen} proves an expected existence of non-zero maps from principal degenerate Whittaker models to admissible representations, and by Proposition \ref{prop:PrinDeg} below, any other degenerate Whittaker model corresponding to the same orbit is mapped onto a  principal degenerate Whittaker model.
Let $\WF(\pi)$ denote the closure in the local field topology of the union of all the orbits in $\WFC(\pi)$. In \S \ref{subsec:PrinDeg} we review the results of \cite{Mat,GS-Gen} and deduce, using Theorem \ref{thm:main}, the following theorem.
\DimaC{
\begin{introthm}[\S \ref{subsec:PrinDeg}]\label{thm:PL}
Let $G$ be a complex reductive group and let  $\pi \in\mathcal{M}(G)$. Let $(S,\varphi)$ be a Whittaker pair such that $\varphi$ is given by Killing form pairing with a principal nilpotent element of the Lie algebra of a Levi subgroup of $G$. Then
$$ \cW_{S,\varphi}(\pi)\neq 0 \Leftrightarrow \varphi\in \WF(\pi).$$
\end{introthm}
}
If $G$ is classical then the set $\WF(\pi)$  is uniquely determined by its intersection  with the set of ``principal Levi" nilpotents $\varphi$, see \cite[Theorem D]{GS-Gen}.
 \DimaE{A result related to Theorem \ref{thm:PL} is proved by Matumoto in \cite{MatENS}. Namely, let  $G$ be a complex reductive group and let  $\pi \in\mathcal{M}(G)$ have regular infinitesimal character. Let $(S,\varphi)$ be a Whittaker pair and assume that the orbit of $\varphi$ is dense in $\WF(\pi)$, and also that it contains a dense subset of the nilradical of the parabolic subgroup defined by $S$. Then $0< \dim \cW_{S,\varphi}(\pi)<\infty$. Matumoto also proves the vanishing of the corresponding higher homology groups.}

For $\F=\R$ a weaker version of Theorem \ref{thm:PL} holds, see Corollary \ref{cor:real} below.

In Section \ref{sec:NewDef} we give choice-free definitions of degenerate Whittaker models. These definitions use the Deligne filtration instead of the Jacobson-Morozov theorem and thus might be suitable for local fields of positive characteristic.

In the global case, instead of degenerate Whittaker models one considers explicit functionals on  automorphic representations defined by integration against a character of a nilpotent subgroup. Such functionals are called Whittaker-Fourier coefficients and \ denoted $\cW\cF_{S,\varphi}(\pi)$.
In Section \ref{sec:Glob} we give the definitions and explain how to adapt our arguments to the global case and deduce the following theorem.

 \begin{introthm}[See \S \ref{sec:Glob}] \label{thm:Glob}
Let $K$ be a number field, let $G$ be the group of adelic points of a reductive group defined over $K$ and $\fg$ be its Lie algebra. Let $\pi$ be an automorphic representation of $G$.

\DimaG{
Let $(S,\varphi) \in \fg\times \fg^*$ be a Whittaker pair. \DimaD{Suppose that $\cW\cF_{S,\varphi}(f)\neq 0$ for some $f\in \pi$. Then  $\cW\cF_{\varphi}(f')\neq 0$ for some $f'\in \pi$.}
}
\end{introthm}
\subsection{Further results for $\GL_n(\F)$}

For $G_n:=\GL_n(\F)$ we show that Theorem \ref{thm:main} allows one to compare degenerate Whittaker models corresponding to any nilpotent orbits $\cO,\cO'\subset \fg_n^*:=\gl(n,\F)^*$ s.t. $\cO\subset \overline{\cO'}$.
\DimaC{To that end we prove the following geometric theorem.

\begin{introthm}[\S \ref{subsec:PfGLmain}] \label{thm:GLmain}
Let $\cO,\cO'\subset \fg_n^*$ be nilpotent coadjoint orbits with $\cO\subset \overline{\cO'}$. Then  there exists a Whittaker pair $(S,\tilde \varphi)$ such that $\tilde \varphi\in  \cO'$ and $\cO$ intersects $\overline{G_S\tilde \varphi}$.
Moreover, $S$ can be chosen to be diagonal with integer eigenvalues and $\tilde \varphi$ can be chosen to be given by the trace pairing with a matrix in Jordan form.
\end{introthm}
\begin{remark}\label{rem:orb}
An analogous statement holds for some pairs of nilpotent orbits $\cO\subset \overline{\cO'}$ in other classical groups (see Remark \ref{rem:GenEmb} below), but not in general. Indeed,  it can be shown that if $\overline{G_S\tilde \varphi}$ intersects a distinguished orbit $\cO$ then $\cO=\cO'$.
\end{remark}

Using Theorem \ref{thm:PL}, \cite{Mat}, and Corollary \ref{cor:real} for archimedean $\F$, and  Theorems \ref{thm:main} and \ref{thm:GLmain}, together with the results of \cite{MW,Var} for non-archimedean $\F$ we deduce the following theorem.

\begin{introthm}[\S \ref{subsec:PfGLOrb}]\label{thm:GLOrb}
Let  $\pi \in\mathcal{M}(G_{n})$. Let $\cO \subset \g_n^{*}$ be a nilpotent orbit. Then we have
\begin{equation}\label{=WWF}
\cW_{\cO}(\pi) \neq 0 \Leftrightarrow \cO \subset \WF(\pi).
\end{equation}
\end{introthm}

\begin{remark}
It would be interesting to know to what extent  this theorem holds for other reductive groups.
For archimedean $\F$, Theorem \ref{thm:PL} and Corollary \ref{cor:real} provide a partial extension. In the non-archimedean case, the failure of Theorem \ref{thm:GLmain} for general groups,  as noted in Remark \ref{rem:orb}, represents the main obstacle for extending our approach. \end{remark}
}

In  \S \ref{sec:PfGL} we give a precise description of the generalized Whittaker spaces in terms  of certain functors $E^k$ introduced  in \cite{AGS,AGS2} in connection with the generalization of the theory of Bernstein-Zelevinsky derivatives to the archimedean setting. We will also use related functors $I^k$ that go in the other direction. We refer to \S \ref{sec:PfGL} below for the precise definitions of both functors.

\begin{introthm}[\S \ref{sec:PfGL}]
\label{thm:GL}  Let $\lambda=(\lambda_1,\dots,\lambda_k)$ be a partition of $n$ and $\cO_{\lambda}\subset\g^*_n$ be the corresponding nilpotent orbit.  Then
\begin{equation}\label{=DerIndIso}
\cW_{\cO_{\lambda}} \simeq I^{\lambda_1}(\cdots I^{\lambda_k}(\C)\cdots),
\end{equation}
where $\C$ denotes the one-dimensional representation of the trivial group $G_0$, and  for $\pi \in\mathcal{M}(G_{n})$ we have
\begin{equation}\label{=DerIso}\cW_{\cO_{\lambda}}(\pi) \simeq (E^{\lambda_k}(\cdots E^{\lambda_1}(\pi)\cdots))^{*}.\end{equation}
\end{introthm}

\DimaB{ Let $\cO \subset \g^*_n$ be a nilpotent orbit.  Let  $\mathcal{M}^{\leq \cO}(G_{n})$ denote the Serre subcategory of $\mathcal{M}(G_{n})$ consisting of representations with wave-front set inside the closure of $\cO$ and let $\mathcal{M}^{< \cO}(G_{n})$  denote the Serre  subcategory of $\mathcal{M}^{\leq \cO}(G_{n})$ consisting of representations with wave-front set not  containing $\cO$. Let $\mathcal{M}^{ \cO}(G_{n}):=\mathcal{M}^{\leq \cO}(G_{n})/\mathcal{M}^{< \cO}(G_{n})$ denote the quotient category.

Using the main results of \cite{AGS,AGS2,GS,GS-Gen}  we obtain

\begin{introcor}[\S \ref{subsec:PfCorGL}]\label{cor:GL} \begin{enumerate}[(i)]
\item \label{GLit:Exact} The functor $\pi \mapsto \cW_{\cO}(\pi)$ defines an exact faithful functor from $\mathcal{M}^{ \cO}(G_{n})$ to the category of finite dimensional
vector spaces.

\item \label{GLit:MaxOrb} If $\pi\in \mathcal{M}^{\leq \cO}(G_{n})$ is an irreducible  unitarizable representation, or a monomial representation, and $\pi \notin \mathcal{M}^{< \cO}(G_{n})$ then $\cW_{\cO}(\pi)$ is one-dimensional.
\end{enumerate}
\end{introcor}
This corollary is new only in the archimedean case, since over $p$-adic fields exactness is well-known and finiteness of dimension is shown in \cite{MW}. It is also shown in \cite{MW} that $\cW_{\cO}(\pi)$ is one-dimensional for all irreducible $\pi\in\mathcal{M}^{ \cO}(G_{n})$. Over archimedean fields, $\cW_{\cO}(\pi)$ is clearly not one-dimensional for any irreducible $\pi$ of finite dimension bigger than one.

Let us also comment that for any irreducible $\pi\in \cM(G_n)$, $\WF(\pi)$ is the closure of a single nilpotent orbit. Over archimedean fields this follows from \cite{BB1,Jos85} and over non-archimedean fields this is proven in \cite{MW}. More generally, it follows from \cite{Jos85} that for any real reductive group $G$ and any irreducible $\pi\in \cM(G)$, all the maximal orbits in $\WF(\pi)$ lie in a single complex nilpotent orbit. An analogous statement is conjectured over p-adic fields but not proven yet.}

We conjecture that the functor $\pi \mapsto \cW_{\cO}(\pi)$ is exact for all reductive groups, and we hope to prove this in the future, generalizing the technique of \cite{AGS2}.

\subsection{The structure of our proofs}

Let us first describe the idea of the proof of Theorem \ref{thm:main} in the case $\varphi=\tilde \varphi$.
We first show that $S$ can be presented as $h+Z$, where $h$ is a neutral element for $\varphi$ and $Z$ commutes with $h$ and $\varphi$. Then we consider a deformation $S_t=h+tZ$, and denote by $\fu_t$ the sum of eigenspaces of $ad(S_t)$ with eigenvalues at least 1. We call a rational number $0< t <1 $ {\it regular} if $\fu_t=\fu_{t+\eps}$ for any small enough rational $\eps$, and {\it critical} otherwise. Note that there are finitely many critical numbers, and denote them by $t_1<\dots<t_n$. Denote also $t_0:=0$ and $t_{n+1}:=1$. For each $t$ we define two subalgebras $\fl_t,\fr_t\subset \fu_t$. Both $\fl_t$ and $\fr_t$ are maximal isotropic subspaces with respect to the form $\omega_{\varphi}$, $\fr_t$ contains all the eigenspaces of $Z$ in $\fu_t$ with positive eigenvalues and $\fl_t$ contains all the eigenspaces with negative eigenvalues. Note that the restrictions of $\varphi$ to $\fl_t$ and $\fr_t$ define characters of these subalgebras.
Let $L_t:=\Exp(\fl_t)$ and $R_t:=\Exp(\fr_t)$ denote the corresponding subgroups and $\chi_\varphi$ denote their characters defined by $\varphi$.
The Stone-von-Neumann theorem implies $$\cW_{S_t,\varphi}\simeq \ind_{L_t}^G(\chi_{\varphi})\simeq \ind_{R_t}^G(\chi_{\varphi}).$$
We show that for any $0\leq i\leq n, \, \fr_{t_i}\subset \fl_{t_{i+1}}$. This gives a natural epimorphism
$$\cW_{S_{t_i},\varphi}\simeq \ind_{L_{t_{i}}}^G(\chi_{\varphi})\onto \ind_{R_{t_i}}^G(\chi_{\varphi})\simeq \cW_{S_{t_{i+1}},\varphi}.$$
Altogether, we get
$$\cW_{h,\varphi}=\cW_{S_{t_0},\varphi}\onto \cW_{S_{t_1},\varphi}\onto \cdots \onto \cW_{S_{t_{n+1}},\varphi}=\cW_{S,\varphi}.$$

If $\varphi \neq \tilde \varphi$, we identify $\fg \simeq \fg^*$ using a non-degenerate invariant form and complete $\varphi$ to an $sl_2$-triple $(e,h,\varphi)$ such that $h$ commutes with $S$. Then we
show, using the Slodowy slice, that the conditions imply that $\tilde \varphi$ is conjugate under $G_S$ to
$\varphi+\varphi'$ with $ad^*(e)(\varphi')=0$. We finish the proof by a deformation argument similar to the case $\varphi = \tilde \varphi$.

\DimaC{For Theorem \ref{thm:PL}, the implication $\cW_{S,\varphi}(\pi)\neq 0\Rightarrow \varphi\in \WF(\pi)$ follows from \cite{Mat}. To prove the other direction, we show that $\cW_{S,\varphi}$ maps onto some principal degenerate  Whittaker model $\cW_{\tilde S,\varphi}$. Thus the theorem follows from the nonvanishing of  $\cW_{\tilde S,\varphi}(\pi)$ (under the condition $\varphi\in \WF(\pi)$), which was shown in
\cite{GS-Gen}. The same argument gives Corollary \ref{cor:real} - an analogous statement for quasisplit real reductive groups.}

For the proof of Theorem \ref{thm:GLmain} we identify $\fg_n^*$ with $\fg_n$ using the trace form, and parameterize nilpotent orbits by partitions. Then we prove the theorem for partitions of length two by an elementary matrix conjugation argument. We finish the proof by induction. The induction argument, however, is not so easy since the statement is not ``transitive''. For any pair of partitions $\lambda \leq \mu$ (where $\leq$ refers to the natural order on partitions which corresponds to the closure order on orbits), we consider two pairs of partitions of two smaller numbers that add up to a number bigger than $n$. Then we take $S'$ and $S''$ corresponding to the two pairs of partitions and force them to coincide on the joint block by adding a scalar matrix to one of them. In this way we obtain a diagonal matrix $S\in \gl_n(\Z)$ that satisfies the requirements of the theorem.

\DimaC{For archimedean $\F$,  Theorem \ref{thm:GLOrb} follows from Theorem \ref{thm:PL}, \cite{Mat}, and Corollary \ref{cor:real}. For non-archimedean $\F$,} it was shown in \cite{MW} that $\cW_{S,\tilde\varphi}(\pi)\neq 0$ for any Whittaker pair $(S,\tilde \varphi)$ such that the orbit of $\tilde\varphi$ is a maximal orbit in $\WFC(\pi)$. If $\cO\subset \WF(\pi)$ then $\cO\subset\overline{\cO'}$ for some maximal orbit $\cO'\in\WFC(\pi)$. Pick $(S,\tilde\varphi)$ that correspond to $\cO,\cO'$ by Theorem \ref{thm:GLmain}. Then, by Theorem \ref{thm:main}, $\cW_{S,\tilde\varphi}(\pi)$ embeds into $\cW_{\cO}(\pi)$ and thus $\cW_{\cO}(\pi)\neq 0$.

We prove Theorem \ref{thm:GL} by induction on $k$. We let $\lambda':=(\lambda_2,\dots,\lambda_k)$, and by the induction hypothesis get $$\cW_{\cO_{\lambda'}} \simeq I^{\lambda_2}(\cdots I^{\lambda_k}(\C)\cdots).$$
Thus, in order to prove \eqref{=DerIndIso} we have to show that
\begin{equation}\label{=IdDer}
\cW_{\cO_{\lambda}} \simeq I^{\lambda_1}(\cW_{\cO_{\lambda}}).
\end{equation}
We note that both sides of the formula are isomorphic to inductions of the same character from two nilpotent subgroups that differ only in the last $\lambda_1$ columns. Then we prove \eqref{=IdDer} by a deformation argument similar to the proof of Theorem \ref{thm:main}.
Finally, \eqref{=DerIso} follows from \eqref{=DerIndIso} by a version of Frobenius reciprocity.

Corollary \ref{cor:GL} follows from \eqref{=DerIso} using the properties of archimedean prederivatives proven in \cite{AGS,AGS2,GS,GS-Gen}.

The proof of Theorem \ref{thm:Glob} is analogous to the proofs of Theorems \ref{thm:main}, \ref{thm:GLmain}. The only difference is that we cannot apply the Stone-von-Neumann theorem since in the global case we consider Whittaker-Fourier coefficients, that are some explicit functionals on an automorphic representation. We replace it by Lemma \ref{lem:Glob}, that is proven by an explicit integral transform followed by a Fourier transform on a compact abelian group. This lemma is in the spirit of \cite[Propositions 7.2 and 7.3]{GRS_Book}.

\DimaC{If $\F$ is archimedean, one can consider more general models, and analogs of Theorems \ref{thm:main}-\ref{thm:GL} will remain valid for them, see Remark \ref{rem:RealGen}.}

\subsection{Acknowledgements}
We thank Avraham Aizenbud, Joseph Bernstein, Wee Teck Gan, Crystal Hoyt, Erez Lapid, Ivan Loseu,  David Soudry, \Dima{Birgit Speh,} and Cheng-Bo Zhu for fruitful discussions, \Dima{and Baying Liu} \DimaE{and the anonymous referee} for useful remarks.

This paper was conceived during the workshop ``Representation Theory and Analysis of Reductive Groups: Spherical Spaces and Hecke Algebras'' at Oberwolfach in January 2014. We would like to thank the organizers of the workshop Bernhard Kroetz, Eric Opdam, Henrik Schlichtkrull, and Peter Trapa,  and the administration of the MFO for inviting us and for the warm and stimulating atmosphere.

D.G. learnt a lot about degenerate Whittaker models from a semester-long seminar in Tel Aviv university organized by Erez Lapid and David Soudry, and on geometry of nilpotent orbits from a spring school at the Weizmann Institute  organized by Anthony Joseph and Anna Melnikov.

D.G. was partially supported by \DimaA{ERC StG grant 637912}, ISF grant 756/12 and a Minerva foundation grant.

\DimaB{Part of the work on this paper was done during the visit of S. Sahi to Israel in summer 2014. This visit was partially funded by the ERC grant 291612.}

\section{Preliminaries}

\subsection{Notation}
For a semi-simple element $S$ and a rational number $r$ we denote by $\fg_r^S$ the $r$-eigenspace of the adjoint action of $S$ and by $\fg_{\geq r}^S$ the sum $\bigoplus_{r'\geq r}\fg_{r'}^S$. We will also use the notation $(\fg^*)_{ r}^S$ and $(\fg^*)_{\geq r}^S$ for the corresponding grading and filtration of the dual Lie algebra $\fg^*$.  For $X\in \fg$ or $X\in \fg^*$ we denote by $\fg_X$ the centralizer of $X$ in $\fg$, and by $G_X$ the centralizer of $X$ in $G$.
We say that an element $h\in \fg$ is \emph{rational semi-simple} if its adjoint action on $\fg$ is diagonalizable with  eigenvalues in $\Q$.

If $(f,h,e)$ is an $\sl_2$-triple, we will say that $e$ is a nil-positive element for $h$, $f$ is a nil-negative element for $h$, and $h$ is a neutral element for $e$. For a representation $V$ of $(f,h,e)$ we denote by $V^e$ the space spanned by the highest-weight vectors and by $V^f$ the space spanned by the lowest-weight vectors.

From now on we fix a non-trivial unitary additive character
\begin{equation}\label{=chi}
\chi:\F\to \mathrm{S}^1
\end{equation} such that if $\F$ is archimedean we have $\chi(x)=\exp(2\pi i \Re(x))$ and if $\F$ is non-archimedean the kernel of $\chi$ is the ring of integers.

\subsection{$\sl_2$-triples}

We will need the following lemma which summarizes several well-known facts about $\sl_2$-triples.

\begin{lemma}[{See \cite[\S 11]{Bou} or \cite{KosSl2}}]\label{lem:sl2}
$\,$
\begin{enumerate}[(i)]
\item Any nilpotent element is the nil-positive element of some $\sl_2$-triple in $\fg$.
\item If $h$ \DimaE{has a nil-positive element} then $e$ is a nil-positive element for $h$ if and only if $e\in \fg^h_2$ and $ad(e)$ defines a surjection $\fg^h_0\onto \fg^h_2$. The set of nil-positive elements for $h$ is open in $\fg^h_2$. 
\item If $e$ is nilpotent then $h$ is a neutral element for $e$ if and only if  $e\in \fg^h_2$ and $h\in \Im(ad(e))$. All such $h$ are conjugate under $G_e$.
\item If $(f,h,e)$ and $(f',h,e)$ are $\sl_2$-triples then $f=f'$.
\item If $(f,h,e)$ is an $\sl_2$-triple and $Z$ commutes with two of its elements then it commutes also with the third one.
\end{enumerate}
\end{lemma}
It is easy to see that the lemma continues to hold true if we replace the nil-positive elements by nil-negative ones (and $\fg_2^h$ by $\fg_{-2}^h$).

\begin{defn}\label{def:nuet}
We will say that $h\in \fg$ is a neutral element for $\varphi\in \fg^*$ if \DimaE{$h$ has a nil-positive element in $\fg$,} $\varphi\in (\fg^*)^h_{-2}$, and the linear map $\fg^h_0\to (\fg^*)^h_{-2}$ given by $x \mapsto ad^*(x)(\varphi)$ is an epimorphism.  Note that if we identify $\fg$ with $\fg^*$ (in a $G$-equivariant way) this property becomes equivalent to $\varphi$ being a nilnegative element for $h$, or $-h$ being a neutral element for $\varphi$. \DimaE{We also say that $0\in \fg$ is a neutral element for $0\in \fg^*$.}
\end{defn}

\subsection{Schwartz induction}

\begin{defn}\label{def:ind}
If $G$ is an $l$-group, we denote by $\Rep^{\infty}(G)$ the category of smooth representations of $G$. If $H\subset G$ is a closed subgroup, and $\pi\in \Rep^{\infty}(H)$ is a smooth representation of $G$ we denote by $\ind_H^G(\pi)$ the smooth compactly-supported induction as in \cite[\S 2.22]{BZ}.

If $G$ is an affine real algebraic group, we denote by $\Rep^{\infty}(G)$ the category of smooth nuclear \Fre representations of $G$ of moderate growth. This is essentially the same definition as in \cite[\S 1.4]{dCl} with the additional assumption that the representation spaces are nuclear (see e.g. \cite[\S 50]{Tre}).
 If $H\subset G$ is a Zariski closed subgroup, and $\pi\in \Rep^{\infty}(H)$ we denote by $\ind_H^G(\pi)$ the Schwartz induction as in \cite[\S 2]{dCl}. More precisely, in \cite{dCl} du Cloux defines a map from the space $\Sc(G,\pi)$ of Schwartz functions from $G$ to the underlying space of $\pi$ to the space $C^{\infty}(G,\pi)$ of all smooth $\pi$-valued functions on $G$ by $f\mapsto \overline f$ where $$\overline f (x)=\int_{h\in H}\pi(h)f(xh)dh,$$
and $dh$ denotes a fixed left-invariant measure on $H$. The Schwartz induction $\ind_H^G(\pi)$ is defined to be the image of this map.
\end{defn}

From now till the end of the subsection let $G$ be either an $l$-group or an affine real algebraic group, and $H'\subset H \subset G$ be (Zariski) closed subgroups.

\begin{lem}[{\cite[Proposition 2.25(b)]{BZ} and \cite[Lemma 2.1.6]{dCl}}]\label{lem:IndSt}
For any $\pi\in \Rep^{\infty}(H')$ we have
$$\ind_{H'}^G(\pi)\simeq \ind_{H}^G\ind_{H'}^{H}(\pi).$$
\end{lem}

\begin{cor}
For any $\pi\in \Rep^{\infty}(H)$ we have
a natural epimorphism $\ind_{H'}^G(\pi|_{H'})\onto \ind_{H}^G(\pi)$.
\end{cor}

\begin{lem}\label{lem:Frob}
Let $\rho \in \Rep^{\infty}(H)$, $\pi\in \Rep^{\infty}(G)$ and let $\pi^*$ denote the dual representation, (endowed with the strong dual topology in the archimedean case).  Then
$$\Hom_{G}(\ind_H^G(\rho),\pi^*) \cong \Hom_H(\rho,\pi^*\Delta_H^{-1}\Delta_G),$$
where $\Delta_H$ and $\Delta_G$ denote the modular functions of $H$ and $G$.
\end{lem}
The non-archimedean case of this lemma follows from \cite[Proposition 2.29]{BZ}.
We prove the archimedean case in Appendix \ref{app:Frob}. We will only use this lemma in the case when $G$ is reductive, $\pi \in \cM(G)$, $H$ is nilpotent and $\rho$ is one-dimensional.
\subsection{Oscillator representations of the Heisenberg group}
\begin{defn}
Let $W_n$ denote the  2n-dimensional $\F$-vector space $(\F^n)^* \oplus \F^n$ and let $\omega$ be the standard symplectic form on $W_n$.
The Heisenberg group $H_n$ is the algebraic group with underlying algebraic variety $W_{n} \times \F$ with the group law given by $$(w_1,z_1)(w_2,z_2) = (w_1+w_2,z_1+z_2+1/2\omega(w_1,w_2)).$$
\Dima{Note that $H_0=\F$.}
\end{defn}


\begin{defn}
Let $\chi$ be the additive character  of $\F$, as in \eqref{=chi}.  Extend $\chi$ trivially to a character of the commutative subgroup $0\oplus \F^n\oplus \F \subset H_n$.
The oscillator representation ${\varpi}_\chi$ is the unitary induction of $\chi$ from $0\oplus \F^n\oplus \F$ to $H_n$. Define the smooth oscillator representation
$\sigma_\chi$ to be the space of smooth vectors in ${\varpi}_\chi$.
\end{defn}

\begin{lemma}\label{lem:OscInd}
$\sigma_{\chi}=\ind_{0\oplus \F^n\oplus \F}^{H_n}(\chi)$
\end{lemma}
\begin{proof}
In the archimedean case we apply the characterization of smooth vectors in a unitary induction given in \cite[Theorem 5.1]{Pou}. By this characterization $\sigma_{\chi}$ can be identified with the space
$$\{f \in C^{\infty}((\F^n)^*) \, \vert \, x^if^{(j)}\in L^2((\F^n)^*) \, \forall i,j\}.$$
This space coincides with the Schwartz space $\Sc((\F^n)^*)$,
which in turn can be identified with $\ind_{0\oplus \F^n\oplus \F}^{H_n}(\chi)$.

In the non-archimedean case let us prove a stronger statement: $\ind_{0\oplus \F^n\oplus \F}^{H_n}(\chi)=\Ind_{0\oplus \F^n\oplus \F}^{H_n}(\chi),$ where $\Ind$ denotes the full smooth induction.
Indeed let $f\in \Ind_{0\oplus \F^n\oplus \F}^{H_n}(\chi),$ and let $f'$ be the restriction of $f$ to $(\F^n)^*\oplus 0\oplus 0$. Since $f$ is smooth, {\it i.e.} fixed by an open compact subgroup $K$ of $H_n$, for any $\varphi \in (\F^n)^*$ and $v\in \F^n\cap K$ we have $\chi(\varphi(v))f'(\varphi)=f'(\varphi)$. This implies that $f'$ has compact support, and thus $f\in \ind_{0\oplus \F^n\oplus \F}^{H_n}(\chi)$.
\end{proof}

\begin{theorem}[Stone-von-Neumann]\label{thm:StvN}
The oscillator representation $\varpi_\chi$ is the only irreducible unitary representation of $H_n$ with central character $\chi$.
\end{theorem}

\begin{cor}\label{cor:StvN}
 Let $L\subset W$ be a Lagrangian subspace. Extend $\chi$ trivially to the abelian subgroup $L\oplus \F \subset H_n$. Then $\ind_{L\oplus \F}^{H_n}\chi \cong \sigma_\chi$.
\end{cor}

\begin{lem}\label{lem:OscDual}
Let $\F$ be non-archimedean. Let $\tau$ be a smooth representation of $H_n$ on which the center acts by the character $\chi$, and let $\widetilde{\tau}$ denote the smooth contragredient. Then $(\Hom_{H_n}(\sigma_\chi,\tau))^*\cong \Hom_{H_n}(\sigma_\chi,\widetilde{\tau})$.
\end{lem}

\begin{proof}
Let $L\subset W$ be a maximal lattice such that $\omega(L,L)\subset \Ker(\chi)$. Then by Theorem \ref{thm:StvN} $\sigma_{\chi}=\ind_{L\times \F}^{H_n}\chi$. By  \cite[2.29]{BZ} $\Hom_{H_n}(\sigma_\chi,\widetilde{\tau})\cong (\tau^*)^L$. Since  $L$ is an open compact subgroup we get $(\tau^*)^L \cong (\tau^L)^*$ and $\Hom_{H_n}(\sigma_\chi,\tau)\cong \tau^L$.
\end{proof}

\subsection{Degenerate Whittaker models}\label{subsec:DefMod}


\begin{defn}
\begin{enumerate}[(i)]
\item A \emph{Whittaker pair} is an ordered pair $(S,\varphi)$ such that $S\in \fg$ is rational semi-simple, and $\varphi\in (\fg^*)^{S}_{-2}$. Given such a Whittaker pair, we define the space of \emph{degenerate Whittaker models} $\cW_{S,\varphi}$ in the following way: let $\fu:=\fg_{\geq 1}^S$.
Define an anti-symmetric form $\omega_\varphi$ on $\fg$ by $\omega_\varphi(X,Y):= \varphi([X,Y])$. Let $\fn$ be the radical of $\omega_\varphi|_{\fu}$. Note that $\fu,\fn$ are nilpotent subalgebras of $\fg$, and $[\fu,\fu]\subset \fg^S_{\geq 2}\subset \fn$.
 Let $U:=\Exp(\fu)$ and $N:=\Exp(\fn)$ be the corresponding nilpotent subgroups of $G$.
Let  $\fn' :=\fn\cap \Ker(\varphi), \, N':=\Exp(\fn')$. \Dima{If $\varphi=0$ we define \begin{equation}\cW_{S,0}:=\ind_{U}^G(\C).\end{equation}
Assume now that $\varphi$ is non-zero.}
Then $U/N'$ has a natural structure of a Heisenberg group, and its center is $N/N'$. Let $\chi_{\varphi}$ denote the unitary character of $N/N'$ given by $\chi_{\varphi}(\exp(X)):=\chi(\varphi(X))$.
Let $\sigma_\varphi$ denote the oscillator representation of $U/N'$ with central character $\chi_{\varphi}$, and $\sigma'_\varphi$ denote its trivial lifting to $U$. Define \begin{equation}\cW_{S,\varphi}:=\ind_{U}^G(\sigma'_\varphi).\end{equation}


\item For a nilpotent element $\varphi\in \fg^*$, define the \emph{generalized Whittaker model} $\cW_{\varphi}$  corresponding to $\varphi$ to be $\cW_{S,\varphi}$, where $S$ is a neutral element for $\varphi$ \Dima{if $\varphi\neq 0$ and $S=0$ if $\varphi=0$.} We will also call $\cW_{S,\varphi}$ \emph{neutral degenerate Whittaker model}. By Lemma \ref{lem:sl2} $\cW_{\varphi}$ depends only on the coadjoint orbit of $\varphi$, and does not depend on the choice of $S$. Thus we will also use the notation $\cW_{\cO}$ for a nilpotent coadjoint orbit $\cO\subset \fg^*$. In \S \ref{sec:NewDef} we reformulate this definition without choosing $S$, but using the Killing form.

\item For $\pi\in \cM(G)$ define the degenerate and generalized Whittaker spaces of $\pi$ by
\begin{equation}\cW_{S,\varphi}(\pi):=\Hom_G(\cW_{S,\varphi},\pi^*) \text{ and }\cW_{\varphi}(\pi):=\Hom_G(\cW_{\varphi},\pi^*).
\end{equation}
\end{enumerate}
\end{defn}

Note that $\cW_{S,\varphi}(\pi)\cong \Hom_{G}(\cW_{S,\varphi},\widetilde{\pi}),$ where $\widetilde{\pi}$ denotes the contragredient representation. In the non-archimedean case this is obvious and in the archimedean case this follows from the Dixmier-Malliavin theorem \cite{DM}.

\begin{lem}\label{lem:WhitFrob}
Let $\fl\subset \fu$ be a maximal isotropic subalgebra and $L:=\Exp(\fl)$. Let $\pi\in \cM(G)$. Then
$$\cW_{S,\varphi}(\pi)\cong \Hom_L(\pi,\chi^{-1}_{\varphi}).$$
\end{lem}
\begin{proof}
By Corollary \ref{cor:StvN} and Lemma \ref{lem:IndSt} we have
$\cW_{S,\varphi}\cong \ind_L^G(\chi_{\varphi}).$ Using Lemma \ref{lem:Frob} we obtain
$$\cW_{S,\varphi}(\pi)\cong \Hom_{G}(\ind_L^G(\chi_{\varphi}),\pi^*)\cong \Hom_L(\chi_{\varphi},\pi^*) \cong \Hom_L(\pi,\chi_{\varphi}^{-1}).$$
\end{proof}

In the case when  $\F$ is non-archimedean and $\pi\in \cM(G)$, slightly different degenerate Whittaker models are considered in \cite{MW}. Namely,
let $U''$ denote the subgroup of $U$ generated by $\Exp(\fg_{>1}^S)$ and the kernel of $\chi_{\varphi}$. Let $\pi_{(U'',\chi_{\varphi})}$ denote the biggest quotient of $\pi$ on which $U''$ acts by the character $\chi_{\varphi}$. Then \cite{MW} considers $\Hom_U(\sigma_{\varphi}, \pi_{(U'',\chi_{\varphi})}$). By Lemma \ref{lem:OscDual} and Frobenius reciprocity we have
\begin{equation}
\cW_{S,\varphi}(\pi) \cong (\Hom_U(\sigma_{\varphi}, \pi_{(U'',\chi_{\varphi})}))^*.
\end{equation}

\begin{remark}
For non-archimedean $\F$ we can define $\cW\cM_{S,\varphi}$ to be the full induction $\Ind_{U}^G\sigma_{\varphi}'$. Since for $L$ as in Lemma \ref{lem:WhitFrob} we have $\sigma_{\varphi}'=\Ind_{L}^U \chi_{\varphi}$, our proof of Theorem \ref{thm:main} will show that under the under the conditions of this theorem we have a $G$-equivariant embedding $\cW\cM_{S,\tilde \varphi} \into \cW\cM_{\varphi}$.
For $\pi \in \cM(G)$ one can define $\cW\cM_{S,\varphi}(\pi):=\Hom_G(\pi,\cW\cM_{S,\varphi})$.
By the Frobenius reciprocity \cite[Theorem 2.28]{BZ} we have $\cW\cM_{S,\varphi}(\pi)=\Hom_L(\pi,\chi_{\varphi})$ which by Lemma \ref{lem:WhitFrob} is isomorphic to $\cW_{S,\varphi}(\pi)$. Thus, all the results of the paper can be reformulated in terms of the full induction.

In order to have an analogous formulation in the archimedean case one needs a notion of full induction of smooth \Fre representations of moderate growth, that will satisfy transitivity of induction, Frobenius reciprocity (as in \cite[Theorem 2.28]{BZ}) and $\sigma_{\varphi}'=\Ind_{L}^G \chi_{\varphi}$. A certain full induction $\cO_M\Ind_H^G(\pi,V)$ is defined in \cite[Definition 2.1.3]{dCl}. It consists of functions of moderate growth from $G$ to $V$ which are equivariant under $H$. It satisfies the first two of our requirements but not the third one. Probably in the suitable notion of full induction the definition of function of moderate growth should take into account the action of $H$ on $V$.
\end{remark}

\begin{remark} \label{rem:RealGen}
If $\F$ is archimedean, one can 
define $\cW_{S,\varphi}$ for any semi-simple $S$ with real eigenvalues in the same way, and the proof of Theorem \ref{thm:main} will be valid for this case without changes.
\end{remark}



\section{Proof of Theorem \ref{thm:main}}\label{sec:PfMain}
\setcounter{lemma}{0}

We will prove in \S \ref{subsec:PfMain2} the following generalization of Theorem \ref{thm:main}.

\begin{thm} \label{thm:main2}
Let $(S,\varphi)$ and $(\widetilde{S},\tilde{\varphi})$ be two Whittaker pairs in $\fg$
such that $\varphi \in \overline{G_{\widetilde{S}}\tilde \varphi}$. Suppose that $\fg_\varphi\cap \fg^S_{\geq 1}\subset \fg^{\widetilde{S}}_{\geq1}$ and that there exists a neutral element $h$ for $\varphi$ such that $h$ commutes with $S$ and $\widetilde{S}$, and  $S-h$ commutes with $\tilde \varphi$.
Then there is a $G$-equivariant epimorphism of $\cW_{S,\varphi}$ onto $\cW_{\widetilde{S},\tilde \varphi}$.
\end{thm}

In order to deduce Theorem \ref{thm:main} we will need the following lemma.

\begin{lemma}[See \S \ref{subsec:PfLemMW}]\label{lem:MW}
Let $\cP$ denote the set of conjugacy classes of Whittaker pairs in $\fg$ and
let $\mathcal{Q}$ denote the set of conjugacy classes of
pairs of elements $\varphi\in \fg^*, Z\in \fg_\varphi$  such that $\varphi$ is nilpotent and $Z$ is rational semi-simple.

Define a map $\mu:\mathcal{Q} \to \cP$ in the following way: for any $q\in \mathcal{Q}$ choose $(Z,\varphi)\in q$, let $h \in \fg_Z$ be a neutral element for $\varphi|_{\fg_Z}$ and define $\mu(q)$ to be the class of the pair $(Z+h,\varphi)$. Then the map $\mu$ is a well-defined  bijection.
\end{lemma}

\begin{proof}
[Proof of Theorem \ref{thm:main}]
By Lemma \ref{lem:MW} there exists a neutral element $h$ for $\varphi$ which commutes with $S$.  Then we have $\fg_\varphi\cap \fg^h_{\geq 1}=0$ . Theorem \ref{thm:main2} applied to the Whittaker pairs $(h,\varphi)$ and $(S,\tilde\varphi)$ implies now that there exists a $G$-equivariant epimorphism of $\cW_{h,\varphi}=\cW_{\varphi}$ onto $\cW_{S,\tilde \varphi}$.
\end{proof}

In the same way we obtain the following corollary of Theorem  \ref{thm:main2} for the case $\varphi=\tilde \varphi$.
\begin{cor} \label{cor:SameE}
Let $(S,\varphi)$ and $(\widetilde{S},\varphi)$ be two Whittaker pairs with the same nilpotent element and commuting semi-simple elements. 
If  $\fg_\varphi\cap \fg^S_{\geq 1}\subset \fg^{\widetilde{S}}_{\geq1}$, then there exists a $G$-equivariant epimorphism of $\cW_{S,\varphi}$ onto
 $\cW_{\widetilde{S},\varphi}$.
\end{cor}
\begin{proof}
By Lemma \ref{lem:MW} applied to the group $G_{\widetilde{S}-S}$ we obtain that there exists a neutral element $h$ for $\varphi$ that commutes with $\widetilde{S}$ and $S$. Thus the corollary follows from Theorem \ref{thm:main2}.
\end{proof}

This corollary enables one to define a preorder on the set of models corresponding to a fixed nilpotent element $\varphi$. Let us describe this preorder more explicitly. Choose a neutral element $h$ for $\varphi$ and let
$\fa$ be a maximal split Cartan subalgebra in $\fg$ that includes $h$. 
Choose a root system $\Sigma$ on $\fa$. 
By Lemma \ref{lem:MW}, if $(S,\varphi)$ is any Whittaker pair, then 
$S$ is conjugate to $h+Z$ for some $Z$ in the stabilizer  $\fa_{\varphi}$ of $\varphi$ in $\fa$.

\begin{defn}\label{def:ord}
Let $X,Y\in \fa_{\varphi}$. 
We say that $X \geq_{\varphi} Y$ if for any $\alp\in \Sigma$ such that $\alp(h)\leq 0$ and $\alp(X)\geq 1-\alp(h)$, we have $\alp(Y)\geq1-\alp(h)$.
\end{defn}

Corollary \ref{cor:SameE} immediately implies the following one.
\begin{cor}\label{cor:ord}
If  $X \geq_{\varphi} Y$  then there exists a $G$-equivariant epimorphism of $\cW_{h+\DimaF{X},\varphi}$ onto $\cW_{h+\DimaF{Y},\varphi}$.
\end{cor}

\begin{remark}
For $\varphi\notin G\tilde \varphi$, the condition on the existence of $h$ cannot be omitted in Theorem \ref{thm:main2}. Indeed, let $\F$ be a $p$-adic field, $\tilde \varphi$ be a regular nilpotent element in $\fg^*_n$ and $\widetilde{S}$ be a neutral element for $\tilde \varphi$. Let $S=\widetilde{S}$. Then for any non-regular nilpotent orbit we can find a representative  $\varphi\in \overline{G_{\widetilde{S}}\widetilde{\varphi}}$. However, for any supercuspidal representation $\pi$ of $G_{n}$ we have $\cW_{S,\varphi}(\pi)=0$ while $\cW_{\widetilde{S},\tilde \varphi}(\pi)\neq0$.
\end{remark}

\begin{remark}
The condition $\varphi\in \overline{G_{\widetilde{S}}\tilde \varphi}$ in Theorem \ref{thm:main2} cannot be replaced by the weaker condition $\varphi\in \overline{G\tilde \varphi}$. Indeed, let $G:=\GL(4,\F)$, where $\F$ is a $p$-adic field. Let $S:=\widetilde{S}:=\diag(3,1,-1,-3)$. Let $\varphi,\tilde \varphi\in (\fg^*)^{\widetilde{S}}_2$ be defined by trace pairing with nilpotent elements in lower-triangular Jordan form with block sizes $(2,2)$ and $(3,1)$ in correspondence. Then $\varphi\in \overline{G\tilde \varphi}$ but $f\notin \overline{G_{\widetilde{S}}\tilde \varphi}$.
Let $\chi$ be a character of $\GL(2,\F)$, $\sigma$ be an irreducible cuspidal representation of  $\GL(2,\F)$ and $\pi:=\chi\times \sigma \in \cM(G)$ be their Bernstein-Zelevinsky product. Then the spaces $\cW_{S,\varphi}(\pi)$ and $\cW_{\widetilde{S},\tilde \varphi}(\pi)$ can be expressed through the Bernstein-Zelevinsky derivatives (see \cite{BZ-Induced}) in the following way: $\cW_{S,\varphi}(\pi)=D^2(D^2(\pi))$ and $\cW_{\widetilde{S},\tilde \varphi}(\pi)=D^1(D^3(\pi))$.
We have $$D^1(\sigma)=0,\, D^2(\sigma)=\C,\, D^1(\chi)=\chi|_{\GL(1,\Q_p)},\, D^1(D^1(\chi))=\C,\, D^2(\chi)=0,$$ and by the Leibnitz rule for Bernstein-Zelevinsky derivatives $$D^2(\pi)= \chi, \, D^2(D^2(\pi))=0,\quad D^3(\pi)=\chi|_{\GL(1,\Q_p)},\, D^1(D^3(\pi))=\C,$$
and thus $\cW_{\widetilde{S},\tilde \varphi}(\pi)=\C$ while  $\cW_{S,\varphi}(\pi)=0$.
\end{remark}

\subsection{Proof of Lemma \ref{lem:MW}}\label{subsec:PfLemMW}
We will need the following lemma.

\begin{lemma}\label{lem:conj}
Let $(f,h,e)$ be an   $\sl_2$-triple in $\g$, let $L$ be its centralizer in $G$ and let $\fl$ be its centralizer in $\fg$. Let $Z_1,Z_2\in \fl$ and suppose that $h+Z_1$ is conjugate to $h+Z_2$ by an element of $G_f$.
Then $Z_1$  is conjugate to $Z_2$ by   an element of $L$.
\end{lemma}

\begin{proof}
Note the Levi decomposition $G_f=LU$, where  $U$ is the nilradical of $G_f$. It is enough to show that if $u\in U,\, X\in \fg_h$ and $ad(u)X\in \fg_h$ then $ad(u)X=X$. This holds since $u=\Exp(Y)$ for some $Y\in (\fg_f)\cap \fg^h_{<0}$ and $[Y,X]\in \fg_h$.
\end{proof}

\begin{proof}[Proof of Lemma \ref{lem:MW}]

We choose a non-degenerate conjugation-invariant symmetric bilinear form on $\fg$ and use it to identify $\fg$ with $\fg^*$. Thus, instead of $\varphi\in \fg^*$ we will consider $f\in \fg$.

To see that $\mu$ is well-defined,   let  $Z,f\in \fg$. Let  $(f,h,e)$ be an   $\sl_2$-triple in $\fg_Z$. Note that any two choices of such a triple are conjugate by $G_Z\cap G_f$. Note also that for any  $g\in G$,
$(ad(g)f,ad(g)h,ad(g)e)$ is an   $\sl_2$-triple in $\fg_{ad(g)Z}$, and $(ad(g)h+ad(g)z,ad(g)f)=ad(g)(h+z,f)\in \cP$. Thus $\mu$ is well-defined.

To see that $\mu$ is onto, let $c\in \cP$  and $(S,f)\in c$. Fix an embedding \DimaF{$\fg\subset \gl(V)$. Let us show that there exists a basis for $V$ in which $S$ is diagonal and $f$ is in Jordan form.
First of all let $\{\lambda_i\}_{i=1}^k$ be all the eigenvalues of $S$,
ordered such that $\lambda_j=\lambda_i-2$ only if $j=i+1$,
 and let $\{W_{i}\}_{i=1}^k$ be the corresponding eigenspaces. Then $f(W_i)\subset W_{i+1}$, and thus $\{W_i\}_{i=1}^k,\{f|_{W_i}\}_{i=1}^{k-1}$ form a representation of a type A quiver. By \cite{Gab2} (see also \cite[Theorem 3.1(2)]{BGP}), any such representation is a direct sum of indecomposable representations in which all the spaces have dimensions 0 or 1. Each of these representations gives a Jordan chain for $f$. Then the union  of these chains is the required basis. With respect to this basis $f$ is in Jordan form and $S$ is diagonal. Thus there exists a diagonal neutral element $-h'$ for $f$ in $\gl(V)$, which then commutes with $S$.
Since $\fg$ is reductive, there exists  a $\fg$-module projection $p:\gl(V) \onto \fg$.} Let $h:=p(h')$. Then $[h,S]=0$. Moreover, $[h,f]=-2f$ and  $h\in \Im(ad(f))$ and thus, by Lemma \ref{lem:sl2}, $-h$ is a neutral element for $f$. Thus $c=\mu(S-h,f)$.

It is left to show that $\mu$ is injective. Let $q,q'\in \mathcal{Q}$ such that $\mu(q)=\mu(q')$, and let $(Z,f)$ and $(Z',f')$ be their representatives.  Then there exist $\sl_2$-triples $(f,h,e)$ in $\fg_Z$ and $(f',h',e')$ in $\fg_{Z'}$, and $g\in G$ such that $ad(g)(f)=f'$ and $ad(g)(h+Z)=h'+Z'$.
Note that $(f',ad(g)h,ad(g)e)$ is an $\sl_2$-triple and thus $ad(g)h$ and $h'$ are conjugate by $G_f$. Thus we can assume that $f=f',h=h'$, and $h+Z$ is conjugate to $h+Z'$ by $G_f$. By Lemma \ref{lem:conj} this implies that $Z$ is conjugate to $Z'$ by $G_f$ and thus $(Z,f)$ is conjugate to $(Z',f')$.
\end{proof}

\subsection{Proof of Theorem \ref{thm:main2}}\label{subsec:PfMain2}

Let $\omega$ denote the anti-symmetric form $\omega_\varphi$ on $\fg$ defined by $\omega(X,Y):=\varphi([X,Y])$.

Our proof is based on the following lemma, which is in the spirit of \cite[Lemma 2.2]{GRS} or \cite[Lemma 7.1]{GRS_Book} or \cite[Lemma A.1]{LapidMao}.
\begin{lemma}\label{lem:main}
Let $\fl,\fr\subset \fg$ be nilpotent subalgebras such that $[\fl,\fr]\subset \fl\cap\fr$, $\omega|_{\fl}=0$, $\omega|_{\fr}=0$  and   the radical of $\omega|_{\fl+\fr}$ is $\fl\cap \fr$. Then $\fl+\fr$ is a nilpotent Lie algebra and \begin{equation}\ind^{\Exp(\fl+\fr)}_{\Exp(\fl)}\chi_\varphi\simeq \ind^{\Exp(\fl+\fr)}_{\Exp(\fr)}\chi_\varphi.\end{equation}
\end{lemma}
\begin{proof}
\Dima{If $\varphi=0$ then $\omega=0$, thus $\fl+\fr=\fl\cap \fr$, $\fl=\fr$ and there is nothing to prove. Now suppose $\varphi\neq 0$ and
denote} $\fk:=\fl\cap\fr\cap \Ker(\chi_\varphi)$. Then $\Exp(\fl+\fr)/\Exp(\fk)$ is the Heisenberg group corresponding to the symplectic form induced by $\omega$ on the space $(\fl+\fr)/(\fl\cap\fr)$. Since $\fl/(\fl\cap\fr)$ and $\fr/(\fl\cap\fr)$ are Lagrangian subspaces, the representations $\ind^{\Exp(\fl+\fr)/\Exp(\fk)}_{\Exp(\fl)/\Exp(\fk)}\chi_\varphi$ and $\ind^{\Exp(\fl+\fr)/\Exp(\fk)}_{\Exp(\fr)/\Exp(\fk)}\chi_\varphi$ are both isomorphic to the oscillator representation $\sigma_\varphi$ of $\Exp(\fl+\fr)/\Exp(\fk)$ with  central character defined by $\chi_\varphi$. Since $\Exp(\fk)$ acts trivially on $\ind^{\Exp(\fl+\fr)}_{\Exp(\fl)}\chi_\varphi$ and $\ind^{\Exp(\fl+\fr)}_{\Exp(\fr)}\chi_\varphi$, we obtain  that they are both isomorphic to the trivial extension of $\sigma_\varphi$ to  $\Exp(\fl+\fr)$.
\end{proof}
\begin{remark}\label{rk:inttransform}
By induction by stages we obtain $\ind^{G}_{\Exp(\fl)}\chi_\varphi\simeq \ind^{G}_{\Exp(\fr)}\chi_\varphi.$ Observe that this isomorphism can be realized explicitly as an integral transform: given $f\in  \ind^{G}_{\Exp(\fl)}\chi_\varphi$ we can define $\check{f}\in  \ind^{G}_{\Exp(\fr)}\chi_\varphi$ simply by setting
\[
\check{f}(g)=\int_{\Exp(\fl\cap\fr) \backslash \Exp(\fr)} f(ng)\, dn.
\]
Then the previous results imply that the map $f\mapsto \check{f}$ defines an isomorphism between these two spaces.
\end{remark}

In the course of our proof we will make several choices and introduce some notation. The reader is welcome to track those on Examples \ref{ex:GLSame} and \ref{ex:GLSmall} below.

Let $z:=S-h$ and $K:=\widetilde{S}-z$. Choose a symmetric bilinear non-degenerate $G$-invariant form on $\fg$ and let $f\in \fg$ correspond to $\varphi$ using this form.
 Let $e$ be the nil-positive element for $h$ and $f$.
Then $e\in \fa:=\fg_z$. Consider the  embedding $\fa^* \hookrightarrow \fg^*$ corresponding to the bilinear form on $\fg$. Let $A:=G_z$.

\begin{lemma}
There exists $\varphi'\in ((\fa^*)^e)^K_{-2}$ such that $\varphi+\varphi'\in G_{\widetilde{S}}\tilde \varphi$. 
\end{lemma}
\begin{proof}

Note that $\fa^*=(\fa^*)^e\oplus ad^*(f)(\fa^*)$. Since $K$ preserves both summands we get
\begin{equation}(\fa^*)^K_{-2}=((\fa^*)^e)^K_{-2}\oplus (ad^*(f)(\fa))^K_{-2}= ((\fa^*)^e)^K_{-2}\oplus ad^*(\fa_K)(\varphi).
\end{equation}
Consider  the map $\nu:A_K\times ((\fa^*)^e)^K_{-2} \to (\fa^*)^K_{-2}$ given by $\nu(g,X):= g(\varphi+X)$. Note that the differential of $\nu$ at the point $(1,0)$ is onto, and thus the image of $\nu$ contains an open neighborhood of $\nu(1,0)=\varphi$. Since $\varphi\in \overline{G_{\widetilde{S}}\tilde \varphi}$, the image of $\nu$ intersects the orbit $G_{\widetilde{S}}\tilde \varphi$.
\end{proof}

Let $\tilde \varphi':=\varphi+\varphi'$. Let $i$ denote the smallest of the $h$-weights of $\varphi'$. If $\varphi'=0$ we take $i$ to be 0.  Note that $i$ is always non-negative.

\begin{lemma}\label{lem:rad}
If $\varphi'\neq 0$ then there exists $X\in \fa_\varphi\cap \fa^K_{2}\cap \fa^h_{-i}$ such that $ \varphi'(X)=1$.
\end{lemma}
\begin{proof}
Let $\varphi'_i$ be the component of $\varphi'$ of weight $i$. There exists $Y\in \fa$ with $ \varphi'_i(Y)=1$. Let $X'\in \fa_{\varphi}$ be the component of $Y$ in the decomposition $\fa=[e,\fa]\oplus \fa_\varphi$. Since $\varphi'_i\in (\fa^*)^e$ , $\varphi'_i$ vanishes on $[e,\fa]$ and thus $\varphi'_i(X')=1$. Decompose $\fa_\varphi$ to joint eigenspaces of the commuting semi-simple operators $h$ and $K$ and let $X$ be the component of $X'$ in $\fa_\varphi\cap \fa^K_{2}\cap \fa^h_{-i}$. Then $\varphi'(X) = \varphi'_i(X)= \varphi'_i(X')=1$.
\end{proof}

Let $Z:=\widetilde{S}-S=K-h\in \fg_\varphi$. For any rational number $0\leq t\leq 1$ define \begin{equation}\label{=ut}
S_t:=S+tZ,\quad \fu_t:=\fg^{S_t}_{\geq 1},\quad \fv_t:=\fg^{S_t}_{> 1},\text{ and }\fw_t:=\fg^{S_t}_{1}. \quad
\end{equation}
\begin{defn}\label{def:crit}
We call $t$ \emph{regular} if $\fu_t = \fu_{t+\eps}$ for any small enough $\eps\in \Q$, or in other words $\fw_t\subset \fg_Z$. If $t$ is not regular we call it \emph{critical}. For convenience, we will say that 0 is critical and 1 is regular.
\end{defn}
Note that there are only finitely many critical numbers.

\begin{lemma}\label{lem:help}
The following results hold:
\begin{enumerate}[(i)]
\item \label{it:OmInv} The form $\omega$ is $ad(Z)$-invariant.
\item \label{it:KerOm} $\Ker \omega = \fg_\varphi=\fg^f\subset \fg^h_{\leq 0}$.
\item \label{it:KerNeg} $\Ker(\omega|_{\fw_t})=\Ker(\omega)\cap \fw_t$. 
\item \label{it:Kerv} $\Ker(\omega|_{\fu_t})=\fv_t\oplus \Ker(\omega|_{\fw_t}) $.
\item \label{it:LW}  $\fw_s\cap\fg_\varphi\subset \fu_{t}$  for any $s<t$.
\end{enumerate}
\end{lemma}
\begin{proof}
\eqref{it:OmInv}: $\omega([Z,a],b)+\omega(a,[Z,b])= \varphi([[Z,a],b]) + \varphi([a,[Z,b]]\rangle= (ad^*(Z)(\varphi))([a,b])=0$.\\
\eqref{it:KerOm}: $a\in \Ker \omega \iff \omega(a,b)=0 \forall b \iff \varphi([a,b])=0 \forall b \iff  (ad^*(a)(\varphi))(b)=0 \forall b \iff ad^*(a)(\varphi)=0 $. \Dima{Thus $\Ker \omega = \fg_\varphi$. Since $\varphi$ is given by pairing with $f$, its stabilizer $\fg_\varphi$ coincides with the space $\fg^f$ that is spanned by the lowest weight vectors.}\\
\eqref{it:KerNeg} $\Ker(\omega|_{\fw_t})=\fw_t\cap \Ker\omega$ since $\omega(\fw_t,\fg^{S_t}_{s})=0$ for any $s\neq -1$. Now $\fg^h_{\leq 0} \cap \fw_t\subset\fg^Z_{>0}$. \\
%
\eqref{it:Kerv} holds since $\omega(\fu_t,\fv_t)=0$.\\
\eqref{it:LW} Let $Y\in \fw_s\cap\fg_\varphi \cap \fg^Z_p$. If $p\geq 0$ then $Y\in   \fu_{t}$. If $p<0$ then $Y\in \fg^S_{\geq1}$. By the conditions of the theorem we have $ \fg^S_{\geq 1}\cap \fg_\varphi\subset \fg^{\widetilde{S}}_{\geq 1}$, thus $Y\in  \fg^S_{ \geq 1}\cap \fg^{\widetilde{S}}_{\geq 1}\subset \fu_t$.
\end{proof}

Choose a Lagrangian $\fm\subset \fg^Z_0\cap \fg^S_{1}$
 and let
 \begin{equation}\label{=lt}
 \fl_t:=\fm+(\fu_t\cap \fg^{Z}_{< 0})+\Ker(\omega|_{\fu_t})\text{ and }\fr_t:=\fm+(\fu_t\cap \fg^{Z}_{> 0})+\Ker(\omega|_{\fu_t}).
\end{equation}

\begin{lemma} \label{lem:key}
\begin{enumerate}[(i)]
\item  \label{it:lrCom} The spaces $\fl_t$ and $\fr_t$ are ideals in $\fu_t$ and $[\fl_t,\fr_t]\subset \fl_t\cap\fr_t$.

\item \label{it:MaxIs} For any $t\geq0$, $\fl_t$ and $\fr_t$ are maximal isotropic subspaces of $\fu_t$.

\item \label{it:Emb} Suppose that $0\leq s< t$, and all the elements of the open interval $(s,t)$ are regular.\\ 
Then $\fr_s\subset\fl_{t}$.
\end{enumerate}
\end{lemma}


\begin{proof}
\eqref{it:lrCom} follows from the inclusion $[\fu_t,\fu_t]\subset \fv_t\subset \fl_t\cap \fr_t$.

\eqref{it:MaxIs} Since $\omega$ is $ad(Z)$-invariant, we see that $\fl_t$ and $\fr_t$ are isotropic. To show that $\fl_t$ is maximal isotropic, let $Y \in \fg^Z_{s}\cap \fu_t$. Let $R$ denote the radical of $\omega|_{\fu_t}$. If $Y\in R$ then $Y\in \fl_t$. If $s\leq 0$ then $Y\in \fl_t$.  If $s>0$ and $Y \notin R$ then there exists $Y'\in \fg^Z_{-s}\cap \fu_t\subset \fl_t$ such that $\omega(Y,Y')\neq 0$. Thus, if we enlarge $\fl_t$ it will stop being isotropic. Now, note that $\omega$ defines a symplectic structure on $\fu_t/R \simeq \fw_t/(\fw_t\cap \fg^Z_0+\fw_t\cap\fg_\varphi )$, and the image of $\fl_t$ in this space is Lagrangian. The image of $\fr_t$ is a complementary isotropic subspace, thus also a Lagrangian, and thus $\fr_t$ itself is maximal isotropic in $\fu_t$.

\eqref{it:Emb}
Note that $\fw_s\cap \fg^{Z}_{> 0}\subset \fv_{t}\subset \fl_{t}$.
Let us show that $\fv_{s} \subset \fl_{t}$. Note that $\fv_s\subset \fu_{t}$, since all the elements in $(s,t)$ are regular.
Let $Y\in \fv_s$ be a joint eigenvector for $ad(S)$ and $ad(Z)$. If $Y \notin \fv_{t}$ then $Y\in \fw_{t}$ and its $Z$-eigenvalue is negative. Thus $Y \in \fl_{t}$. Now by Lemma \ref{lem:help} \eqref{it:LW} we get $\fw_s\cap \fg_\varphi \subset \fu_{t}\cap \fg_\varphi\subset \fl_{t}$, and from Lemma \ref{lem:help} \eqref{it:KerOm}-\eqref{it:Kerv} this implies $\Ker(\omega|_{\fu_s})\subset \fl_{t}$. Altogether we get  $\fr_s\subset\fl_{t}$.
\end{proof}


Define
\begin{equation}\label{=lt'}
 \fl_t':=\fl_t\cap \Ker(\varphi') \text{ and } \fr_t':=\fr_t\cap \Ker(\varphi').
 \end{equation}

\begin{lem}\label{lem:lrprime}
\begin{enumerate}[(i)]
\item \label{it:inc} For any $0\leq s<t$ such that all the numbers in the open interval $(s,t)$ are regular  we have $\fr'_s\subset \fl'_t$.
\item \label{it:ideal} For $0\leq t<1$ both $\fl'_t$ and $\fr'_t$ are subalgebras of $\fu_t$, $[\fl'_t,\fr'_t]\subset \fl'_t\cap \fr'_t$ and for $0\leq t<(i+1)/(i+2)$ we have $\fl'_t=\fl_t$ and $\fr'_t=\fr_t$. \item \label{it:lag} $\fl'_t\cap \fr'_t$ is the radical of $\omega|_{\fl'_t+\fr'_t}$
\end{enumerate}
\end{lem}
\begin{proof}
Part \eqref{it:inc} follows immediately from Lemma \ref{lem:key}\eqref{it:Emb}. \\
For part \eqref{it:ideal} note that $i\geq 0$, $\varphi'\in (\fg^*)^{S_t}_{\geq i - t(i+2)}$ and thus  $\Ker(\varphi') \supset \fg^{S_t}_{> -i + t(i+2)}$. Thus, for $t<1$ we have $[\fu_t,\fu_t]\subset \Ker(\varphi')$ and for $0\leq t<(i+1)/(i+2)$ we have $\fu_t\subset \Ker(\varphi')$. Since $[\fu_t,\fu_t] \subset \fv_t\subset \fl_t\cap\fr_t$ we obtain \eqref{it:ideal}.\\
For part \eqref{it:lag} we can assume that $t\geq (i+1)/(i+2)$ and $\varphi'\neq 0$. Then, by Lemma \ref{lem:rad}, there exists $X\in \fg_\varphi\cap \fu_t$ with $\varphi'(X)=1$. Thus, the image of $\fl'_t$ in $\fu_t/(\fg_\varphi\cap \fu_t)$ coincides with the image of $\fl_t$, and the image of $\fr'_t$ coincides with the image of $\fr_t$. Since those images are Lagrangian, $\fl'_t\cap \fr'_t$ is the radical of $\omega|_{\fl'_t+\fr'_t}$.
\end{proof}

Lemmas \ref{lem:lrprime} and \ref{lem:main} imply that
\begin{equation}\cW_{S,\varphi}\simeq \ind_{\Exp(\fr_0)}^{G}\chi_\varphi \simeq \ind_{\Exp(\fr'_0)}^{G}\chi_\varphi, \quad \ind_{\Exp(\fl'_{t})}^{G}\chi_\varphi\simeq \ind_{\Exp(\fr'_{t})}^{G}\chi_\varphi
\end{equation}
 and for $s<t$ such that all the numbers in $(s,t)$ are regular we have
  \begin{equation}
 \ind_{\Exp(\fr'_s)}^{G}\chi_\varphi\onto \ind_{\Exp(\fl'_{t})}^{G}\chi_\varphi.
\end{equation}
We are now ready to prove Theorem \ref{thm:main2}.

\begin{proof}[Proof of Theorem \ref{thm:main2}]
Let $0<t_1<\dots<t_n$ be all the critical numbers.
Note that  $\fr'_{t_n}$ is an isotropic subalgebra of $\fu_1$, and that it is also isotropic with respect to the form $\omega_{\varphi+\varphi'}(X,Y):= (\varphi+\varphi')( [X,Y])$. Let $\fq$ be a maximal isotropic subspace of $\fu_1$ with respect to this form. It includes $\fv_1$ and thus is necessary a subalgebra. Note that
\begin{equation}
\ind_{\Exp(\fr'_{t_n})}^{G}\chi_{\varphi}=\ind_{\Exp(\fr'_{t_n})}^{G}\chi_{\varphi+\varphi'} \onto \ind_{\Exp(\fq)}^{G}\chi_{\varphi+\varphi'} \simeq \cW_{\widetilde{S},\tilde \varphi'}\simeq \cW_{\widetilde{S},\tilde \varphi}.
\end{equation}

Thus
\begin{multline}
$$
\cW_{S,\varphi}\simeq \ind_{\Exp(\fr'_0)}^{G}\chi_\varphi\onto \ind_{\Exp(\fl'_{t_1})}^{G}\chi_\varphi\simeq \ind_{\Exp(\fr'_{t_1})}^{G}\chi_\varphi \onto\\ \dots  \onto \ind_{\Exp(\fl'_{t_n})}^{G}\chi_\varphi\simeq \ind_{\Exp(\fr'_{t_n})}^{G}\chi_\varphi  \onto \cW_{\widetilde{S},\tilde \varphi}.
$$
\end{multline}
 \end{proof}

\begin{remark}\label{rem:IntTran} Observe that, in the above proof, the map $\ind_{\Exp(\fr'_{t_{i}})}^{G}\chi_\varphi\onto \ind_{\Exp(\fl'_{t_{i+1}})}^{G}\chi_\varphi$ is simply given by integration over $\Exp(\fr'_{t_{i}})\backslash \Exp(\fl'_{t_{i+1}})$. From this, and the observation given in Remark \ref{rk:inttransform}, we see that the $G$-equivariant epimorphism promised in Theorem \ref{thm:main2} is given by a series of integral transforms followed, if necessary, by conjugation by an element of $G$. 
\end{remark}

\DimaD{
\begin{remark}\label{rem:Equivar}
Let $M$ denote the joint centralizer of $\varphi, \tilde \varphi, S, \tilde S,$ and $h$. Then a central extension $\tilde M$ of $M$ acts naturally on the oscillator representations $\sigma_{\varphi}$ and  $\sigma_{\tilde\varphi}$, and thus also on the degenerate Whittaker models $\cW_{S,\varphi}$ and $\cW_{\widetilde{S},\tilde \varphi}.$ It is easy to see that the constructed epimorphism $\cW_{S,\varphi}\onto \cW_{\widetilde{S},\tilde \varphi}$ intertwines these actions.
\end{remark}

}

Let us now present two examples for the elements and subalgebras defined in the course of the proof. Let $G:=\GL(4,\F)$ and define $\varphi$ by $\varphi(X):=\tr(X(E_{21}+E_{43}))$, where $E_{21}$ and $E_{43}$ are elementary matrices. Let $h$ be the diagonal matrix $\diag(1,-1,1,-1)$ and $S:=h$.

\begin{example}\label{ex:GLSame}
Let $\tilde \varphi:=\varphi$, $K:=\diag(3,1,-1,-3), \, z:=0 $. Then $Z=\diag(2,2,-2,-2), S_t=\diag(1+2t,-1+2t,1-2t,-1-2t)$ and the weights of $S_t$ are as follows:
$$\left(
   \begin{array}{cccc}
     0 & 2 & 4t & 4t+2 \\
     -2 & 0 & 4t-2 & 4t \\
     -4t & -4t+2 & 0 & 2 \\
     -4t-2 & -4t & -2 & 0 \\
   \end{array}
 \right).$$
The  critical numbers are $1/4$ and $3/4$. For $t\geq3/4$ we get the principal degenerate Whittaker model. We have $\fr'_t=\fr_t,\fl'_t=\fl_t$ for all $t$.
The above system of inclusions of $\fr_{0}\subset \fl_{1/4}\sim \fr_{1/4}\subset \fl_{3/4}=\fr_{3/4}$  is:
$$
\left(
   \begin{array}{cccc}
     0 & - & 0 & - \\
     0 & 0 & 0 & 0 \\
     0 & - & 0 & - \\
     0 & 0 & 0 & 0 \\
   \end{array}
 \right) \subset \left(
   \begin{array}{cccc}
     0 & - & a & - \\
     0 & 0 & 0 & a \\
     0 & * & 0 & - \\
     0 & 0 & 0 & 0 \\
   \end{array}
 \right)\sim
\left(
   \begin{array}{cccc}
     0 & - & * & - \\
     0 & 0 & 0 & * \\
     0 & 0 & 0 & - \\
     0 & 0 & 0 & 0 \\
   \end{array} \right)
\subset
\left(
   \begin{array}{cccc}
     0 & - & - & - \\
     0 & 0 & * & - \\
     0 & 0 & 0 & - \\
     0 & 0 & 0 & 0 \\
    \end{array} \right)
   $$
   Here, both $*$ and $-$ denote arbitrary elements. $-$ denotes the entries in $\fv_t$ and $*$ those in $\fw_t$. The letter $a$ denotes an arbitrary element, but the two appearances of $a$ denote the same numbers.
   The  passage from $\fl_{1/4}$ to $\fr_{1/4}$ is denoted by $\sim$. At  $3/4$ we have $\fl_{3/4}=\fr_{3/4}. $
\end{example}

Let us now give an example in which $\varphi$ and $\tilde \varphi$ are not equal and not conjugate.

\begin{example}\label{ex:GLSmall}
Identify $\fg\simeq \fg^*$ using the trace form and let

$$
\tilde \varphi=\left(
   \begin{array}{cccc}
     0 & 0 & 1 & 0 \\
     1 & 0 & 0 & 1 \\
     0 & 0 & 0 & 0 \\
     0 & 0 & 1 & 0 \\
   \end{array}
 \right), \quad
e=\left(
   \begin{array}{cccc}
     0 & 1 & 0 & 0 \\
     0 & 0 & 0 & 0 \\
     0 & 0 & 0 & 1 \\
     0 & 0 & 0 & 0 \\
   \end{array}
 \right),  \quad
  \varphi'= \left(
   \begin{array}{cccc}
     0 & 0 & 1 & 0 \\
     0 & 0 & 0 & 1 \\
     0 & 0 & 0 & 0 \\
     0 & 0 & 0 & 0 \\
   \end{array}
 \right)$$

 Let $\widetilde{S}:=\diag(0,-2,2,0), \, Z=\diag(-1,-1,1,1)$.
Note that $\fw_0=\fw_1=0$ and the only critical value of $t$ is $1/2$. The sequence of subalgebras 
 from the proof of Theorem \ref{thm:main2} is 
$$ \fu_0=l'_0=\fl'_{1/2}=\left(
   \begin{array}{cccc}
     0 & * & 0 & - \\
     0 & 0 & 0 & 0 \\
     0 & * & 0 & * \\
     0 & 0 & 0 & 0 \\
   \end{array}
 \right) \sim \fr'_{1/2} = \left(
   \begin{array}{cccc}
     0 & * & 0 & 0 \\
     0 & 0 & 0 & 0 \\
     a & * & 0 & * \\
     0 & -a & 0 & 0 \\
   \end{array}
 \right)\subset \left(
   \begin{array}{cccc}
     0 & * & 0 & 0 \\
     0 & 0 & 0 & 0 \\
     * & * & 0 & * \\
     0 & * & 0 & 0 \\
   \end{array}
 \right)=\fu_1.$$
\end{example}

\subsection{Principal degenerate Whittaker models and proof of theorem \ref{thm:PL}}\label{subsec:PrinDeg}


\DimaF{
In the discussion below $\mathfrak{a}$ will denote a maximal split toral
subalgebra of $\mathfrak{g}$, we will write $\Sigma\left(  \mathfrak{a,g}%
\right)  $ for the corresponding (restricted) root system, $\Sigma^{+}\left(
\mathfrak{a,g}\right)  $ for a choice of positive roots, and $\Delta\left(
\mathfrak{a,g}\right)  $ for the corresponding system of simple roots.

\begin{definition}
We say that a rational semisimple  $S\in \fg$ is \emph{principal} if there exists
 an $\mathfrak{a}$ containing $S$, and a simple subsystem
$\Delta=\Delta\left(  \mathfrak{a,g}\right)  \subset\Sigma\left(
\mathfrak{a,g}\right)  $  such that $\alpha\left(  S\right)  =2$ for all
$\alpha\in\Delta$. We say that a Whittaker pair $(S,\varphi)$ is \emph{principal} if $S$ is principal. A \emph{principal degenerate Whittaker model} is the  degenerate Whittaker model corresponding to a principal Whittaker pair.
\end{definition}

We fix a non-degenerate invariant bilinear form on $\fg$, which allows us to identify nilpotent elements in $\fg$ and $\fg^*$. Thus we may  equally well apply the above terminology to ``Whittaker pairs" $(S,f) \subset \fg \times \fg$, such that $S$ is rational semisimple and $[S,f]=-2f$.

Note that a principal element $S$ uniquely determines both $\mathfrak{a}$,  and  the simple system
$\Delta=\Delta\left(  \mathfrak{a,g}\right)  $. Thus there is a bijection
between principal elements and simple systems which we denote by %
$
S\mapsto\Delta_{S},\text{ }\Delta\mapsto S_{\Delta}.
$

If $\left(  S,f\right)  $ is principal then setting $\Delta=\Delta_{S}$, we
can write $f$ uniquely in the form
\begin{equation}
f=%
{\textstyle\sum\nolimits_{\alpha\in\Delta}}
Y_{\alpha}\text{ }.\label{=PL}%
\end{equation}
for some $Y_{\alpha}\in\mathfrak{g}_{-\alpha}$. Conversely if $f$ is of the
form (\ref{=PL}) for \emph{some} $\Delta$, then we will say that $f$ is a \emph{PL}
nilpotent and $\Delta$ is \emph{compatible} with $f$. In this case $\left(
S_{\Delta},f\right)  $ is a principal Whittaker pair, and we define the $\Delta$-support of $f$ to be
\[
supp_{\Delta}\left(  f\right)  =\left\{  \alpha\in\Delta\mid
Y_{\alpha}\neq0\right\}.
\]
If $supp_{\Delta}\left( f\right)  =\Delta$ we say that $f$ is a \emph{principal nilpotent element}. Note that this notion is weaker than the similar notion defined in \cite[VIII.11.4]{Bou}. However, if $G$ is quasi-split then both notions are equivalent to the notion of regular nilpotent element.

Note that $f$ is a \emph{PL}
nilpotent if and only if it is a principal nilpotent element for a Levi subgroup of $G$. For the general linear groups, every orbit includes such an element. For complex classical groups, all such orbits are described in \cite[\S 6]{GS-Gen} in terms of the corresponding partitions.  For complex exceptional groups, these are the orbits with non-parenthetical Bala-Carter labels.


\begin{lemma}\label{lem:2pairs}
Let $\left(  S_{1},f\right)  $ and $\left(  S_{2},f\right)  $ be two Whittaker
pairs with the same $f$, and let $h$ be neutral for $f$. Then there exist
$g_{1},g_{2}\in G_{\varphi}$ and $Z_{1},Z_{2}\in\mathfrak{g}_{h} \cap \fg_{f}$ such that
\[
g_{1}\cdot S_{1}=h+Z_{1},\quad g_{2}\cdot S_{2}=h+Z_{2},\quad\left[
Z_{1},Z_{2}\right]  =0.
\]

\end{lemma}

\begin{proof}
By Lemma \ref{lem:MW} we can find $g_{1},g_{2},Z_{1},Z_{2}$ satisfying all the stated
properties, except perhaps the commutativity $\left[  Z_{1},Z_{2}\right]  =0$. Since $\mathfrak{g}_{h} \cap \fg_{f}$ is the centralizer of an $\sl_2$-triple, it is reductive.
Since $Z_{1},Z_{2}%
\in\mathfrak{g}_{h} \cap \fg_{f}$ are rational semisimple, we can move them to a common Cartan
subspace after further conjugation by elements of $G_h\cap G_{f}$.
\end{proof}

\begin{proposition}\label{prop:PLroot}
Let $\left(  S,f\right)  $ be a Whittaker pair such that $f$ is a PL
nilpotent. Then there exist a maximal split toral subalgebra $\mathfrak{a}$,
and a simple system $\Delta\subset\Sigma^{+}\subset\Sigma\left(
\mathfrak{a,g}\right)  $ such that

\begin{enumerate}[(a)]
\item $\Delta$ is compatible with $f.$

\item $\mathfrak{a}$ contains $S.$

\item $\mathfrak{a}$ contains a neutral element $h$ for $f.$

\item If $\alpha\in\Sigma$ satisfies $\alpha\left(  h\right)  \leq0$ and
$\alpha\left(  S\right)  >0$ then $\alpha\in\Sigma^{+}$.
\end{enumerate}
\end{proposition}

\begin{proof}
Let us first choose any simple system $\Delta\left(  \mathfrak{b,g}\right)  $
compatible with $f$. By Lemma \ref{lem:2pairs} there
exists $g$ in the centralizer $G_{f}$ of $f$ such that $g\cdot S\in
\mathfrak{b}$. Now the action of $g^{-1}$ carries $\mathfrak{b}$ to a maximal
toral subalgebra $\mathfrak{a}$ and $\Delta\left(  \mathfrak{b,g}\right)  $ to
a simple system $\Delta^{\prime}=\Delta^{\prime}\left(  \mathfrak{a,g}\right)
$ which satisfies (a) and (b). Next note that if $S^{\prime}=S_{\Delta
^{\prime}}$ then $\left(  S^{\prime},f\right)  $ is a Whittaker pair, so by
Lemma \ref{lem:MW} there is a neutral element $h$ for $f$ that commutes with
$S_{\Delta^{\prime}}$; but this forces $h\in\mathfrak{a}$ and thus (c) holds.

If $supp_{\Delta^{\prime}}\left( f\right)  =\Delta^{\prime}$ then we set
$\Delta=\Delta^{\prime}$; in this case we have $h-S$ is central and (d) holds vacuously.
Now suppose $supp_{\Delta^{\prime}}\left(  f\right)  \subsetneq\Delta^{\prime
}$. Then we will show how to modify $\Delta^{\prime}$ to obtain a new system
$\Delta\left(  \mathfrak{a,g}\right)  $ such that (a)--(d) are satisfied. For
this let us write
\[
Z=S-h,\quad Z^{\prime}=S^{\prime}-h,\quad h_{\varepsilon}=Z+\varepsilon
Z^{\prime}+\varepsilon^{2}h.
\]
where $\varepsilon>0$ is chosen sufficiently small so that for all $\alpha
\in\Sigma\left(  \mathfrak{a,g}\right)  $ we have
\begin{align}
\alpha\left(  Z\right)    & >0\implies\alpha\left(  Z\right)  >\varepsilon
\left\vert \alpha\left(  Z^{\prime}\right)  \right\vert +\varepsilon
^{2}\left\vert \alpha\left(  h\right)  \right\vert \label{=eps1}\\
\alpha\left(  Z^{\prime}\right)    & >0\implies\alpha\left(  Z^{\prime
}\right)  >\varepsilon\left\vert \alpha\left(  h\right)  \right\vert
\label{=eps2}%
\end{align}
We claim that $h_{\varepsilon}$ is regular in the sense that %
$
\alpha\left(  h_{\varepsilon}\right)  \neq0\text{ for all }\alpha.
$
If $\alpha\left(  Z\right)  >0$ this follows from (\ref{=eps1}), if
$\alpha\left(  Z\right)  <0$ we simply replace $\alpha$ by $-\alpha$. If
$\alpha\left(  Z\right)  =0$ but $\alpha\left(  Z^{\prime}\right)  \neq0$,
then this follows analogously from (\ref{=eps2}). Finally, if $\alpha\left(
Z\right)  =\alpha\left(  Z^{\prime}\right)  =0$ then $\alpha\left(
h_{\varepsilon}\right)  =\varepsilon^{2}\alpha\left(  h\right)  $ and we must
show that $\alpha\left(  h\right)  \neq0$, but this follows from the
regularity of $S^{\prime}=h+Z^{\prime}$.

This means that we can define a positive root system as follows
\[
\Sigma^{+}\left(  \mathfrak{a,g}\right)  =\left\{  \alpha\in\Sigma\mid
\alpha\left(  h_{\varepsilon}\right)  >0\right\}  .
\]
Let $\Delta$ be the corresponding simple system; we will show that
$
supp_{\Delta^{\prime}}\left(  f\right)  \subset\Delta.
$
This implies that $\Delta$ is compatible with $e$, so that (a) holds. To prove that
$
supp_{\Delta^{\prime}}\left(  f\right)  \subset\Delta
$ suppose $\alpha$ belongs to $supp_{\Delta^{\prime}}\left(
f\right)  $. Then we have
$
\alpha\left(  Z\right)  =\alpha\left(  Z^{\prime}\right)  =0,\quad
\alpha\left(  h_{\varepsilon}\right)  =\varepsilon^{2}\alpha\left(  h\right)
=2\varepsilon^{2}>0
$
This means that $\alpha\in\Sigma^{+}\left(  \mathfrak{a,g}\right)  $ and it
remains to show that we cannot write
\begin{equation}
\alpha=\beta+\gamma\label{=simple}%
\end{equation}
where $\beta,\gamma\in\Sigma^{+}\left(  \mathfrak{a,g}\right)  $. Now if
$\beta\left(  Z\right)  >0$ or $\beta\left(  Z^{\prime}\right)  >0$ then
(\ref{=simple}) is impossible by (\ref{=eps1}) and (\ref{=eps2}). Thus we may
assume $\beta\left(  Z\right)  =\beta\left(  Z^{\prime}\right)  =0$ and thus
$
\beta\left(  S^{\prime}\right)  =\beta\left(  Z^{\prime}\right)  +\beta\left(
h\right)  =2>0,
$
and similarly $\gamma\left(  S^{\prime}\right)  >0$. But $\alpha$ is a simple
root in the positive system defined by $S^{\prime}$, so (\ref{=simple}) cannot
hold.

Finally we verify that $\Sigma^{+}\left(  \mathfrak{a,g}\right)  $ satisfies
(d). Thus suppose $\alpha\in\Sigma\left(  \mathfrak{a,g}\right)  $ satisfies
$\alpha\left(  h\right)  \leq0$ and $\alpha\left(  S\right)  >0,$ then we
have
$
\alpha\left(  Z\right)  =\alpha\left(  S\right)  -\alpha\left(  h\right)  >0.
$
By (\ref{=eps1}), this implies
$
\alpha\left(  h_{\varepsilon}\right)  =\alpha\left(  Z\right)  +\varepsilon
\alpha\left(  Z^{\prime}\right)  +\varepsilon^{2}\alpha\left(  h\right)  >0.
$
Thus we have $\alpha\in\Sigma^{+}\left(  \mathfrak{a,g}\right)  $ as desired.
\end{proof}

}

\begin{proposition}\label{prop:PrinDeg}
Suppose that $\varphi\in \fg^*$ can be completed to a principal Whittaker pair. Then any degenerate Whittaker model $\cW_{S,\varphi}$ can be mapped onto some principal degenerate Whittaker model $\cW_{\widetilde{S},\varphi}$.
\end{proposition}

\DimaF{
\begin{proof}
By Corollary \ref{cor:ord}
it suffices to find a maximally split toral subalgebra
$\mathfrak{a}$ and containing a neutral element $h$ for $\varphi$, and a
principal element $\tilde{S}$ such that $\tilde{S}-h\in\mathfrak{a}_{\varphi}$
and $S-h\geq_{\varphi}\tilde{S}-h$ in the sense of Definition \ref{def:ord}, i.e.%
\[
\alpha\left(  h\right)  \leq0\text{ and }\alpha\left(  S\right)  \geq
1\implies\alpha\left(  \tilde{S}\right)  \geq1\text{ for all }\alpha\in
\Sigma\left(  \mathfrak{a,g}\right)  .
\]

Let $f\in \fg$ be the nilpotent element corresponding to $\varphi$. Then $(S,f)$ satisfies the conditions of Proposition \ref{prop:PLroot}.
Let $\mathfrak{a},h,\Sigma^{+},\Delta$ be as in the proposition, and let $\tilde{S}=S_{\Delta}$. Then  any $\alp \in \Sigma\left(  \mathfrak{a,g}\right)$ with $
\alpha\left(  h\right)  \leq0\text{ and }\alpha\left(  S\right)  \geq1
$
lies in  $\Sigma^{+}.$ Since $\tilde{S}$ is principal, we deduce
that $\alpha\left(  \tilde{S}\right)  \geq2.$
\end{proof}
}
\DimaC{
\subsubsection{Archimedean case}\label{subsec:GS}

For a smooth representation $\pi$ of a real reductive group $G$ one can define one more invariant, which we denote $\cV(\pi)$ and call the annihilator variety of $\pi$. It is sometimes called the associated variety of the annihilator of $\pi$. It is defined to be the set of zeros in $\fg_{\C}^*$ of the ideal in the symmetric algebra $S(\fg_{\C})$, which is generated by the symbols of the  annihilator ideal of $\pi$ in the universal enveloping algebra $U(\fg_{\C})$.
It follows from \cite[Theorem 8.4]{Vo91} and \cite{SV} that $\cV(\pi)$ is the Zariski closure of $\WF(\pi)$ in $\fg_\C^*$. Note that if $G$ is a complex reductive group or $G=\GL_n(\R)$
we have $\WF(\pi)= \cV(\pi)\cap \fg^*$.

We will use the following theorems
\begin{thm}[{\cite[Corollary  4]{Mat}}]\label{thm:Mat}
Let $\pi$ be a smooth representation of $G$, let $\cO\subset \fg^*$ be a nilpotent orbit and suppose that $\cW_{\cO}(\pi)\neq 0.$ Then $\cO\subset \cV(\pi)$.
\end{thm}

\begin{thm}[{\cite[Theorem B]{GS-Gen}}]\label{thm:GS}
Let $(S,\varphi)$ be a principal degenerate Whittaker pair for $G$. Let
$\pi\in \cM(G)$ such that $\varphi \in \WF(\pi)$.
Then
\begin{enumerate}[(i)]
\item If $G$ is a complex group or $G=\GL_n(\R)$  then $\cW_{S,\varphi}(\pi)\neq 0$.
\item If $G$\ is quasisplit then there exists $g\in G_{\C}$ such that $ad(g)$ preserves $\fg$ and $\cW_{ad(g)(S),ad(g)(\varphi)}(\pi)\neq 0$.
\end{enumerate}
\end{thm}

In fact, the theorems in \cite{Mat,GS-Gen} are stronger than the versions we state here.

\begin{proof}[Proof of Theorem \ref{thm:PL}]
Let $(S,\varphi)$ be as in the theorem. Then $\varphi$ can be completed to a principal Whittaker pair, and Proposition \ref{prop:PrinDeg} implies that the degenerate Whittaker model $\cW_{S,\varphi}$ can be mapped onto some principal degenerate Whittaker model $\cW_{\widetilde{S},\varphi}$. By Theorem \ref{thm:main}, $\cW_{\varphi}$ maps onto $\cW_{S,\varphi}$. Thus we have $\cW_{\widetilde{S},\varphi}(\pi) \into \cW_{S,\varphi}(\pi)\into \cW_{\varphi}(\pi)$. Together with Theorems \ref{thm:Mat},\ref{thm:GS} we get
$$\varphi \in \WF(\pi) \Rightarrow \cW_{\widetilde{S},\varphi}(\pi)\neq 0 \Rightarrow  \cW_{S,\varphi}(\pi)\neq 0 \Rightarrow  \cW_{\varphi}(\pi)\neq 0 \Rightarrow \varphi \in \WF(\pi).$$
\end{proof}

In the same way we get the following statement for real reductive groups.

\begin{cor}\label{cor:real}
Suppose that  $G$\ is quasisplit  and let  $\pi \in\mathcal{M}(G)$.
\DimaG{
Let $\varphi\in \WF(\pi)$ be a PL
nilpotent, and let $(S,\varphi)$ be a Whittaker pair. Then there exists $g\in G_{\C}$ such that $ad(g)$ preserves $\fg$ and $\cW_{ad(g)(S),ad(g)(\varphi)}(\pi)\neq 0$.}
\end{cor}
}

\section{General linear groups}\label{sec:GL}

\subsection{Notation}
Let us first introduce some notation. A \emph{composition} $\eta$ of $n$ is a sequence of natural (positive) numbers $\eta_1,\dots,\eta_k$ with $\sum\eta_i=n$. The length of $\eta$ is $k$. A  partition $\lambda$ is a composition such that $\lambda_1\geq \lambda_2\geq \dots\geq \lambda_k$.
For a composition $\eta$ we denote by $\eta^\geq$ the corresponding partition.
A partial order on partitions of $n$ is defined by \begin{equation}
\lambda\geq \mu \text{ if }\sum_{i=1}^j \lambda_i\geq \sum_{i=1}^j \mu_i\text{ for any }1\leq j \leq \mathrm{length}(\lambda), \mathrm{length}(\mu).
\end{equation}

We will use the notation $\diag(x_1,\dots,x_k)$ for diagonal and block-diagonal matrices. For a natural number $k$ we denote by $J_k\in \fg_k$ the \emph{lower}-triangular Jordan block of size $k$, and by $h_k$ the diagonal matrix $h_k:=\diag(k-1,k-3,\dots,1-k)$. For a composition $\eta$ we denote \begin{equation}J_{\eta}:=\diag(J_{\eta_1},\dots, J_{\eta_k})\in \fg_n \text{ and }h_{\eta}:= \diag(h_{\eta_1},\dots, h_{\eta_k})\in \fg_n.\end{equation}
Note that $[h_{\eta},J_{\eta}]=-2J_{\eta}$ and $(J_{\eta},h_{\eta})$ can be completed to an $\sl_2$-triple.

 Identify $\fg_n^*$ with $\fg_n$ using the trace form. Denote by $\cO_{\eta}$ the orbit of $J_{\eta}$,
  and also the corresponding orbit in $\fg_n^*$.
 By the Jordan theorem all nilpotent orbits are of this form.
It is well known that $\cO_{\eta}\subset \overline{\cO_{\gam}}$ if and only if $\eta^{\geq} \leq \gam^{\geq}$.
Let $T_n\subset G_n$ denote the subgroup of diagonal matrices and $\ft_n\subset \g_n$ the subalgebra of diagonal matrices.

\subsection{Proof of Theorem \ref{thm:GLmain}}\label{subsec:PfGLmain}

Let $E_{ij}$ denote the elementary matrix with 1 in the $(i,j)$ entry and zeros elsewhere.

\begin{lemma}\label{lem:TwoBlocks}
For any $p,q,r \in \Z_{\geq 0}$ with $p\geq r$ there exists a diagonal matrix $S\in \ft_m(\Z),$ where $m=p+q+r$, and a regular nilpotent $X\in \fg_{q+r}$ such that $[S,J_{(p+q,r)}]=-2J_{(p+q,r)}$ and  $\diag(J_{p},X)\in \overline{(G_m)_SJ_{(p+q,r)}}$.
\end{lemma}
\begin{proof}
If $r=0$ or $q=0$ we take $X:=J_{q+r}$, $S:=h_{(p+q,r)}$ and note that $\diag(J_{p},X)\in \overline{T_mJ_{(p+q,r)}} \subset \overline{(G_m)_SJ_{(p+q,r)}}$. Assume now $q,r>0$ and let $$F :=J_{(p+q,r)},S:=\diag(h_{p+q},h_r+(r+q-p)\Id_r) \in \fg_m,$$ $$\text{and }g:=(\Id_m+E_{p-r+1,m-r+1})(\Id_m+E_{p-r+2,m-r+2})\cdot \dots \cdot(\Id_m+E_{p,m})\in G_m.$$
Note that $S_i:=S_{ii}=p+q-(2i-1)$ for $1\leq i \leq p+q$ and $S_i=2r+3q+p-(2i-1)$ for $p+q+1\leq i \leq m$. Thus $S_{m-r+j}=S_{p-r+j}$ for $1 \leq j \leq r$ and thus $g$ commutes with $S$. Note also that $F':=Ad(g)(F)= F+E_{p+1,m}$.
Conjugating $F'$ by a suitable diagonal matrix we can obtain $F'-(1-t)E_{p+1,p}\in (G_m)_S F$ for any $t\in \F^{\times}.$ Letting $t$ go to zero, we get that $$f:=F'-E_{p+1,p}=F+E_{p+1,m}-E_{p+1,p}\in \overline{(G_m)_S F}.$$
Finally, it is easy to see that $f= \diag(J_{p},X)$ for a regular nilpotent $X\in \fg_{q+r}$.
\end{proof}

\begin{lemma}\label{lem:part}
Let $\lambda,\mu$ be partitions of $n$ with $\lambda \geq \mu$. Then there exists an index $i\leq \length(\lambda)$ such that $\lambda_i\geq \mu_i\geq \lambda_{i+1}$. Here, if $i= \length(\lambda)$ we take $\lambda_{i+1}=0$.
\end{lemma}
\begin{proof}
We prove by induction on $\length(\lambda)$. If $\length(\lambda)=1$ take $i=1$. For the induction step,
assume $\length(\lambda)\geq2$ and the lemma holds for all shorter partitions. If $\mu_1\geq \lambda_2$ take $i:=1$. Otherwise, consider the partitions $\lambda'=(\lambda_1+\lambda_2 -\mu_1,\lambda_3,\dots)$ and $\mu'=(\mu_2,\mu_3,\dots)$. Note that these are indeed partitions and $\lambda'\geq \mu'$. Thus, by the induction hypothesis there exists $j$ such that $\lambda'_j\geq \mu'_j\geq \lambda'_{j+1}$. If $j>1$ take $i:=j+1$. If $j=1$ then $\mu_2\geq \lambda_3$ and we also have $\lambda_2>\mu_1\geq\mu_2$. Thus we can take $i:=j+1=2$ in this case as well.
\end{proof}

We are now ready to prove Theorem \ref{thm:GLmain}. Let us reformulate it in terms of partitions.

\begin{thm}
Let $\lambda,\mu$ be partitions of $n$. Then $\lambda \geq \mu$ if and only if there exists $S\in \ft_n(\Z)$ such that
$[S,J_{\lambda}]=-2J_{\lambda}$ and $\overline{(G_n)_S J_{\lambda}}$ intersects $\cO_{\mu}$.
\end{thm}
\begin{proof}
We prove the lemma by induction on $n$. The base case $n=1$ is obvious.
For the induction step, assume that the lemma holds for all $n' < n$.
By Lemma \ref{lem:part} there exists an index $i$ with $i\leq \length(\lambda)$ such that $\lambda_i\geq \mu_i\geq \lambda_{i+1}$.
Let $$r:=\lambda_{i+1}, \, q:=\mu_i-\lambda_{i+1},\, p:=\lambda_i+\lambda_{i+1}-\mu_i, \text{ and } m:=\lambda_{i}+\lambda_{i+1}.$$
Now consider the partition $\lambda'$ obtained by replacing the blocks $\lambda_i$ and $\lambda_{i+1}$ by a single block $p$, and the partition $\mu'$ obtained from $\mu$ by omitting the block $\mu_i$. Note that both are indeed partitions of $n-\mu_i$ and that $\lambda'\geq \mu'$. Thus, by the induction hypothesis there exists $S''\in \ft_{n-\mu_i}(\Z)$ such that $[S'',J_{\lambda'}]=-2J_{\lambda'}$ and $\overline{(G_{n-\mu_i})_{S''}J_{\lambda'}}$ intersects $\cO_{\mu'}$. Choose $S' \in \ft_m(\Z)$ and $X\in \fg_{\mu_i}$ using Lemma \ref{lem:TwoBlocks}.
Consider the matrix $Z'$ formed by taking the first $p$ elements on the diagonal of $S'$ and the matrix $Z''$ formed by taking the $p$ elements number $\lambda_1+\dots+\lambda_{i-1}+1,\dots, \lambda_1+\dots+\lambda_{i-1}+p$ on the diagonal of $S''$. Note that $Z'-Z''$ is a diagonal matrix that commutes with $J_{p}$ and thus equals $c\Id_p$ for some integer $c$. Replacing $S'$ by $S'-c\Id_m$ we can assume that $Z'=Z''$ and thus there exists $S\in \ft_n(\Z)$ that
includes both $S'$ and $S''$ as diagonal submatrices.

Let us show that $S$ satisfies the conditions of the lemma.
Let $a:=\sum_{j=1}^{i-1}\lambda_j$ and $d:=\sum_{j=i+2}^{\length(\lambda)}\lambda_j$.
Define an embedding of $G_{m}$ into $G_n$ by $\iota_1(g):=\diag(\Id_{a}, g, \Id_{d})$.
Define an embedding of $G_{n-\mu_i}$ into $G_n$  by $$\iota_2\left(
                                                               \begin{array}{cc}
                                                                 A & B \\
                                                                 C & D \\
                                                               \end{array}
                                                             \right)
:=\left(
                                                               \begin{array}{ccc}
                                                                 A & 0 & B \\
                                                                 0 & \Id_{\mu_i} & 0 \\
                                                                 C & 0 & D \\
                                                               \end{array}
                                                             \right),
$$
where $A\in \Mat(a,a,\F),\, B\in \Mat(a,d+p,\F),\, C\in \Mat(d+p,a,\F),\, D\in  \Mat(d+p,d+p,\F).$
Let $d\iota_1:\fg_m\into \fg_n$ and  $d\iota_2:\fg_{n-\mu_i}\into \fg_n$ be the differentials of $\iota_1,\iota_2$.
These embeddings map the centralizers of $S'$ and $S''$ into the centralizer of $S$. 
Let $Y:=d\iota_1(\diag(0,X))+d\iota_2(J_{\lambda'})\in \fg_n$.
Since $\diag(J_{p},X)\in \overline{(G_m)_{S'}J_{(\lambda_i,\lambda_{i+1})}}$ , we have
$Y\in \overline{(G_n)_S J_{\lambda}}$. Since  $\overline{(G_{n-\mu_i})_{S''}J_{\lambda'}}$ intersects $\cO_{\mu'}$, $\overline{(G_n)_S Y}$ intersects $\cO_{\mu}$ and thus $\overline{(G_n)_S J_{\lambda}}$ intersects $\cO_{\mu}$.
\end{proof}

\DimaC{
\subsection{Proof of Theorem \ref{thm:GLOrb}}\label{subsec:PfGLOrb}

In the archimedean case the theorem follows from Theorem \ref{thm:Mat} and Corollary \ref{cor:real}. Thus we assume here that $\F$ is non-archimedean.
Let $\varphi\in \fg^{*}$ and let $\nu:\F^{\times}\to G$ be an algebraic group morphism (defined over $\F$) such that $ad^*(\nu(t))\varphi=t^2\varphi$. Let $S:=d\nu(1)\in \fg$. Following \cite{MW} define $\cW_{\nu,\varphi}:=\cW_{S,\varphi}$.

\begin{thm}[{\cite[Proposition I.11, Theorem I.16 and Corollary I.17]{MW}, and \cite[Proposition 1 and Theorem 1]{Var}}] \label{thm:MW}
Let $\varphi,\nu$ be as above. Let $\pi\in \cM(G)$.
\begin{enumerate}[(i)]
\item If $\cW_{\nu,\varphi}(\pi)\neq 0$ then $\varphi\in \WF(\pi)$.
\item If $\varphi$ belongs to a maximal orbit $\cO\in\WFC(\pi)$ then $\cW_{\nu,\varphi}(\pi)\neq 0$ and its dimension equals the coefficient of $\cO$ in $\WFC(\pi)$.
\end{enumerate}
\end{thm}

%

\begin{proof}[Proof of Theorem \ref{thm:GLOrb}]
Let $\pi\in \cM(G_n)$. Theorem \ref{thm:MW} implies that if $\cW_{\cO}(\pi)\neq 0$ then $\cO\subset \WF(\pi)$.
Suppose now that  $\cO\subset \WF(\pi)$, {\it i.e.} there exists $\cO'\in \WFC(\pi)$ such that $\cO\subset \overline{\cO'}$. Let $\tilde \varphi\in\cO'$ and $S=\diag(\{h_i\})\in \gl_n(\Z)$ be as in Theorem \ref{thm:GLmain}.  Define $\nu:\F^{\times}\to G_n$ by $\nu(t):=\diag(t^{h_i})$. Then $ad^{*}(\nu(t))\tilde \varphi=t^{-2}\tilde \varphi$ and $\cW_{\nu,\tilde \varphi}=\cW_{S,\tilde \varphi}$.  By Theorem \ref{thm:MW} $\cW_{S,\tilde \varphi}(\pi)\neq 0$ and by Theorems \ref{thm:main} and \ref{thm:GLmain}  we have an epimorphism $\cW_{\cO}\onto\cW_{S,\tilde \varphi}$, hence $\cW_{S,\tilde \varphi}(\pi)$ embeds into $\cW_{\cO}(\pi)$, and thus $\cW_{\cO}(\pi)\neq 0$.
\end{proof}

\begin{rem}\label{rem:GenEmb}
One can show that if 
 $\mu$ is obtained from $\lambda$ by taking some parts apart, or  by replacing two parts of the same parity by two equal parts, then  $\cW_{\cO_{\mu}}$ maps onto $\cW_{\cO_{\lambda}}$. This follows from Theorem \ref{thm:main} by taking $S:=h_{\lambda}, \, \tilde \varphi(X):=\tr(XJ_{\lambda}), \text{ and } \varphi(X):=\tr(XJ_{\mu})$. However, this does not extend to arbitrary $\mu \leq \lambda$. For example, if $\lambda=(4,1)$ and $\mu=(3,2)$ then $(\fg_n^*)^{h_{\lambda}}_{-2}$ does not intersect $\cO_{\mu}$.

For the symplectic groups one can show that if $\mu$ is obtained from $\lambda$ by replacing two parts of the same parity by two equal parts, then
one can map the generalized Whittaker model corresponding to an orbit with partition $\mu$ onto the generalized Whittaker model corresponding to an orbit with partition $\lambda$.
\end{rem}
}
\subsection{Definition of derivatives and proof of Theorem \ref{thm:GL}}\label{sec:PfGL}


The notion of derivative was first defined in \cite{BZ-Induced} for smooth
representations of $G_{n}$ over non-archimedean fields and became a
crucial tool in the study of this category. In \cite{AGS}
this construction was extended to the archimedean case.

The definition of derivative is based on the \textquotedblleft
mirabolic\textquotedblright\ subgroup $P_{n}$ of $G_{n}$ consisting of
matrices with last row $(0,\dots,0,1)$.  The unipotent radical of this subgroup
is an $\left( n-1\right) $-dimensional linear space that we denote $V_{n}$,
and the reductive quotient is $G_{n-1}$. We have a natural isomorphism $P_n = G_{n-1} \ltimes V_n$.  The group $G_{n-1}$ has 2 orbits on
$V_{n}$ and hence also on the dual group $V_{n}^{\ast }$: the zero and the non-zero orbit.
The stabilizer in $G_{n-1}$ of a non-trivial character of $V_n$ is isomorphic to $P_{n-1}$.

Let $\psi_{n}$ be the standard non-trivial unitary character of $V_{n}$, given by $$\psi _{n}(x_1,\dots,x_{n-1}):=\chi(x_{n-1}),$$
where $\chi$ is the fixed additive character of $\F$, as in \eqref{=chi}. We will also denote by $\psi_n$ the corresponding character of the Lie algebra $\fv_n$.
For all $n$ and for all smooth representations $\pi $ of ${P}_{n}$, we define
$$\pi(V_n,\psi_n):=\Span\{\pi(a) v - \psi_n(a)v \, : \, v \in \pi, \, a \in V_{n}\},$$
and we put
\begin{equation}
 \Phi^-(\pi):=
 \begin{cases}
 \pi/\overline{\pi(V_n,\psi_n)} &\mbox{if }\F \text{ is archimedean} \\
  \pi/\pi(V_n,\psi_n) &\mbox{if }\F \text{ is non-archimedean} \end{cases} \end{equation}

If $\F$ is non-archimedean,  our definition of $\Phi^-$ coincides with the one in \cite[\S 5.11]{BZ}. It differs from the  definition in \cite{BZ-Induced} by the twist by the character $|\det|^{1/2}$.

For a smooth representation $\pi$ of $G_n$ we define a representation $E^k(\pi)$ of $G_{n-k}$ by \begin{equation}
E^k(\pi):=((\Phi^-)^{k-1}(\pi|_{P_{n}}))|_{G_n}.
\end{equation}
We call it \emph{the $k$-th pre-derivative of $\pi$}.

Define also a functor $\Phi^+_c:\Rep^{\infty}(P_n)\to \Rep^{\infty}(P_{n+1})$  by
\begin{equation}
\Phi_c^+(\pi)=\ind^{P_{n+1}}_{P_{n}\ltimes V_{n+1}}(\pi\boxtimes \psi_{n+1}),
\end{equation}
where $\Rep^{\infty}$ denotes the category of smooth representations as in Definition \ref{def:ind}, and $I^k:\Rep^{\infty}(G_n)\to \Rep^{\infty}(G_{n+k})$ by
\begin{equation}
I^k(\pi):=\ind_{P_{n+k}}^{G_{n+k}}((\Phi_c^+)^{k-1}(\ind^{P_{n+1}}_{G_{n}}\pi)).
\end{equation}

\begin{lem}\label{lem:AdjDer}
Let $\lambda$ be a partition of $n$ and let $\pi\in \cM(G_n)$. Then  $$\Hom_{G_n}(I^{\lambda_1}(I^{\lambda_2}(\dots I^{\lambda_k} (\C)\dots),\widetilde{\pi})=(E^{\lambda_k}(\dots(E^{\lambda_1}(\pi)\dots)))^{*}.$$
\end{lem}
The proof of this lemma is analogous to the proof of Lemma \ref{lem:WhitFrob}.

If $\F$ is archimedean then the space $\pi(V_n,\psi_n)$ has the same closure as the space $\pi(\fv_n,\psi_n)$ defined analogously using the Lie algebra action, since both closures equal to the joint kernel of all $(V_n,\psi_n)$-equivariant continuous functionals on $\pi$.
It is shown in \cite{AGS2} that if $\pi\in \cM(G_n)$ then $\pi(\fv_n,\psi_n)$ is closed. Moreover, it is shown that for any $i$, the space $(\Phi^-)^i(\pi)(\fv_n,\psi_n)$ is closed in this case. Thus for $\pi\in \cM(G_n)$ our definition of $E^k$ coincides with the functor $\tilde E^k$ used in \cite{AGS,AGS2}. It differs from the functor $E^k$ used in \cite{AGS,AGS2} by the twist by the character $|\det|^{(k-1)/2}$. \DimaB{Note though that for non-admissible smooth $\pi$ our definition of $E^k$ might differ from the functor $\tilde E^k$ used in \cite{AGS,AGS2}.}


Let us now start proving Theorem \ref{thm:GL}.
Let $\lambda$ be a partition of $n$ and $\eta$ be the inverse reordering of $\lambda$.
Let $f:=J_{\eta}$, $h:=h_{\eta}$ and let $e\in \fg_n$ be the unique element such that $(f,h,e)$ is an $\sl_2$ triple.
 We will prove Theorem \ref{thm:GL}   by induction on $k:=\length(\eta)$.
Let $Z$ be a diagonal matrix with first $n-\eta_k$ entries equal to zero, and last $\eta_k$ entries equal to $\eta_k+\eta_{k-1}$. Define $\varphi\in \fg^*$ by $\varphi(X):=\tr(fX)$ and let $\omega:=\omega_{\varphi}$.
For $0\leq t \leq 1$ let $S_t:=h+tZ$ and define $\fu_t,\fl_t$ and $\fr_t$ as in formulas (\ref{=ut},\ref{=lt}) in \S \ref{subsec:PfMain2} (note though that our $\fl'_t,\fr'_t$ differ from the ones in formula \eqref{=lt'}). Let $\fa$ denote the stabilizer of the standard basis vector $e_{n-\eta_k+1}$ and define $\fl'_t:=\fa\cap \fl_{t}, \, \fr'_t:=\fa \cap \fr_{t}$.

\begin{lem}
\begin{enumerate}[(i)]
\item \label{it:PrimeCom}$\fl'_t$ and $\fr'_t$ are subalgebras of $\fu_t$ and $[\fl'_t,\fr'_t]\subset\fl'_t\cap \fr'_t$.
\item \label{it:0a} $\fu_0\subset \fa$ and thus $\fl_0=\fl'_0$.
\end{enumerate}
\end{lem}
\begin{proof}
Part \eqref{it:PrimeCom} follows from Lemma \ref{lem:key}\eqref{it:lrCom} and the fact that $\fa$ is a subalgebra of $\fg$.

For part \eqref{it:0a} note that the $h$-weight of $e_{n-\eta_k+1}$ is the maximal weight inside the standard representation, and thus any element of $\fu_0$ annihilates it.
\end{proof}

The next lemma follows from the structure of the lowest-weight vectors in a tensor product of irreducible representations of $\sl_2$.
\begin{lemma}\label{lem:RepDecomp}
Let $\sigma$ and $\tau$ be two irreducible representations of an $\sl_2$-triple $(f,h,e)$ with $\dim \sigma \geq \dim \tau$. Let $L:=\Hom_{\F}(\sigma,\tau)=\sigma^*\otimes \tau$.
Let $L^{0}$ denote the annihilator in $L$ of the highest-weight vector in $\sigma$, and $L^f$ denote the space spanned by all the lowest-weight vectors in $L$. Then $L=L^0\oplus L^f$.
\end{lemma}

\begin{lemma}\label{lem:uDecomp}
$\fu_t=\fu_t\cap \fa \oplus (\fu_t)_\varphi$.
\end{lemma}
\begin{proof}
First decompose $\fu_t=(\fu_t)_{\leq 0}^Z \oplus (\fu_t)_{> 0}^Z$ and note that $(\fu_t)_{\leq 0}^Z\subset \fu_0$ and thus $(\fu_t)_{\leq 0}^Z\subset \fa$ and $(\fu_t)_{\leq 0}^Z\cap (\fu_t)_\varphi\subset (\fu_0)^f=0$. Now let $V$ denote the standard representation of $\fg_n$ and consider the decomposition $V=V_1\oplus \dots \oplus V_k$, where each $V_i$ is the subspace spanned by the basic vectors with indices from $1+\sum_{j=1}^{i-1}\eta_{j}$ to $\sum_{j=1}^{i}\eta_{j}$. Note that $V_i$ is an irreducible representation of the $\sl_2$ triple $(f,h,e)$ of dimension $\eta_i$. Then $(\fu_t)_{> 0}^Z=\oplus_{i=1}^{k-1}\Hom_{\F}(V_k,V_i)$.
By Lemma \ref{lem:RepDecomp} we have $(\fu_t)_{> 0}^Z=(\fu_t)_{> 0}^Z\cap \fa \oplus ((\fu_t)_{> 0}^Z)_\varphi$.
\end{proof}

\begin{cor}\label{cor:Rad}
The radical of the restriction $\omega|_{\fl'_t+\fr'_t}$ is $\fl'_t\cap \fr'_t$.
\end{cor}
\begin{proof}
By Lemma \ref{lem:key}\eqref{it:MaxIs} we have $\Rad(\omega|_{\fl_t+\fr_t})=\fl_t\cap \fr_t$ and thus $\Rad(\omega|_{\fl'_t+\fr'_t}) \supset \fl'_t\cap \fr'_t$ . Now, $(\fu_t)_\varphi \subset \fl'_t\cap \fr'_t $ and by Lemma \ref{lem:uDecomp} the image of $\fl'_t$ in $\fu_t/(\fu_t)_\varphi$ coincides with the image of $\fl_t$ and the image of $\fr'_t$ in $\fu_t/(\fu_t)_\varphi$ coincides with the image of $\fr_t$. Thus $\Rad(\omega|_{\fl'_t+\fr'_t})=\fl'_t\cap \fr'_t$.
\end{proof}
Recall  Definition \ref{def:crit} of regular and critical numbers.
\begin{lem}\label{lem:l=r}
For $s<t$ with no critical numbers in the open interval $(s,t)$ we have $\fr'_s=\fl'_t$.
\end{lem}
\begin{proof}
By Lemma \ref{lem:key}\eqref{it:Emb} we have $\fr_s\subset \fl_t$ and thus $\fr'_s\subset \fl'_t$. For the inverse inclusion note that $$\fl'_t=\fm\oplus (\fv_t)^Z_{\leq 0} \oplus (\fv_t)^Z_{> 0} \oplus (\fw_t)^Z_{<0}.$$
We have $\fm\subset \fr'_s$; $(\fv_t)^Z_{\leq 0}\oplus (\fw_t)^Z_{<0} \subset \fv_s\subset \fr'_s$. Since there are no critical elements in $(s,t)$ we also have $(\fv_t)^Z_{> 0} \subset (\fu_s)^Z_{> 0}\subset \fr'_s$.
\end{proof}
\begin{proof}[Proof of Theorem \ref{thm:GL}]
Corollary \ref{cor:Rad} and Lemmas \ref{lem:l=r} and \ref{lem:main} imply \begin{equation}
\cW_{h,\varphi}^{G_n}\simeq \ind_{\Exp(\fl_0)}^{G_n}(\chi_\varphi)\simeq \ind_{\Exp(\fr'_1)}^{G_n}(\chi_\varphi).
\end{equation}
Let $\eta^{-}:=(\eta_1,..,\eta_{k-1})$ and consider the corresponding elements $h_{\eta^-}\in\g_{n-\eta_k}$ and $\varphi_{\eta^-}\in\g_{n-\eta_k}^*$.
Note that
\begin{equation}\ind_{\Exp(\fr'_1)}^{G_n}(\chi_\varphi) \simeq I^{\eta_k}(\cW^{G_{n-\eta_k}}_{h_{\eta^-},\varphi_{\eta^-}})= I^{\eta_k}(\cW^{G_{n-\eta_k}}_{\cO_{\eta}})
\end{equation}
The isomorphism \eqref{=DerIndIso} now follows by induction on $\length(\eta)$.
The isomorphism \eqref{=DerIso} follows from \eqref{=DerIndIso} using Lemma \ref{lem:AdjDer}.
\end{proof}

\subsection{Proof of Corollary \ref{cor:GL}}\label{subsec:PfCorGL}
For non-archimedean $\F$, \DimaB{exactness is well-known, the rest of part \eqref{GLit:Exact}  follows from Theorem \ref{thm:MW}, and part \eqref{GLit:MaxOrb} follows from \cite[\S II.2]{MW}.} Thus we assume that $\F$ is archimedean.

\subsubsection{Preliminaries on pre-derivatives}\label{subsubsec:PrelDer2}

The highest non-zero pre-derivative of $\pi\in \cM(G_n)$  plays a special
role. It has better properties than the other derivatives. In particular it is also admissible. The index of the
highest non-zero pre-derivative is called the depth of $\pi$.
As shown in \cite{GS,AGS,GS-Gen} the depth also equals the maximum among the first parts of the partitions in the orbits in  $\WF(\pi)$.
The following theorem summarizes the main results of \cite{AGS,AGS2}.

\begin{thm}
\label{thm:Der} Let $\mathcal{M}^{d}(G_{n}) \subset \mathcal{M}(G_{n})$  denote the subcategory of representations of depth $%
\leq d$. Then

\begin{enumerate}[(i)]

\item \label{Derit:Exact} The functor $E^k : \cM(G_n) \to
\Rep^{\infty}(P_{n-k+1})$ is exact for any $1\leq k \leq n$.

\item \label{Derit:Adm} $E^d$ defines a functor $\cM^d(G_n) \to
\mathcal{M}(G_{n-d})$.



\item \label{Derit:Prod} Let $n=n_1+\cdots +n_d$ and let $\chi_i$ be characters of $G_{n_i}$. Let $\pi= \chi_1 \times \cdots \times \chi_d \in \cM^{d}(G_n)$ denote the corresponding monomial representation. Then
    $$E^d(\pi)\cong((\chi_1)|_{G_{n_1-1}} \times  \cdots  \times (\chi_d)|_{G_{n_d-1}}) $$
\item \label{Derit:A} If $\tau$ is an irreducible unitary representation of $G_n$ and $\tau^{\infty}$ has depth $d$ then $E^d(\tau^{\infty})\cong (A\tau)^{\infty}$, where $A\tau$ denotes the (irreducible, unitary) adduced representation defined in \cite{Sahi-Kirillov} (cf. \cite{Bar}). 

\end{enumerate}
\end{thm}

For non-archimedean $\F$, the theorem follows from \cite{BZ-Induced} since $E^d$ coincides with the highest Bernstein-Zelevinsky derivative considered in \cite{BZ-Induced}.

For archimedean $\F$, $\WF(E^k(\pi))$ is calculated in \cite{GS-Gen}. In particular,  \cite[Theorem 5.0.5]{GS-Gen}      implies the following result.

\begin{thm}
\label{thm:DerDepth} Let $\F$ be archimedean and let $\pi\in\mathcal{M}(G_{n})$. Suppose
that $\WF(\pi)=\overline{{\mathcal{O}}_{(n_{1},...,n_{k})}}$
with $n_{1}\geq... \geq n_{k}$. Then $depth(\pi)=n_1$ and $\WF(E^{n_{1}}(\pi))=\overline{{\mathcal{O}}_{n_{2},...,n_{k}}}$.
\end{thm}

\subsubsection{Proof of Corollary \ref{cor:GL}}

Let $\lambda$ be a partition of $n$ and $\cO_{\lambda}$ be the corresponding nilpotent orbit. Denote $\cW_{\lambda}(\pi):=(E^{\lambda_k}(\dots E^{\lambda_1}(\pi)\dots))^*$. We use Theorem \ref{thm:GL} and identify $\cW_{\cO_{\lambda}}(\pi)$ with $\cW_{\lambda}(\pi)$.
\DimaB{
We prove the theorem by induction on $n$, using Theorems \ref{thm:Der} and \ref{thm:DerDepth}. For the base of the induction we note that $\cM(G_0)$ is the category of finite-dimensional vector spaces, and monomial or irreducible representations of $G_0$ are one-dimensional. For the induction step, let $\mu:=(\lambda_2,\dots,\lambda_k)$ and note that by Theorem \ref{thm:Der}
$E^{\lambda_1}$ is an exact functor from  $\cM^{\leq \cO_{\lambda}}(G_n)$ to $\cM(G_{n-\lambda_1})$. By Theorem \ref{thm:DerDepth}, it maps
$\cM^{\leq \cO_{\lambda}}(G_n)$ to $\cM^{\leq \cO_{\mu}}(G_{n-\lambda_1})$, and $E^{\lambda_1}(\pi)\in \cM^{<\cO_{\mu}}(G_{n-\lambda_1})$ if and only if $\pi \in \cM^{< \cO_{\lambda}}(G_n)$. Thus, by the properties of quotient categories (see \cite[\S III.1]{Gab}), $E^{\lambda_1}$ defines an exact and faithful functor from $\cM^{\cO_{\lambda}}(G_n)$ to $\cM^{\cO_{\mu}}(G_{n-\lambda_1})$.  Theorem \ref{thm:Der} also implies that if $\pi\in \cM^{\leq \cO_{\lambda}}(G_n)$ is monomial then so is $E^{\lambda_1}(\pi)$ and if $\pi$ is irreducible unitarizable, then so is $E^{\lambda_1}(\pi)$. By the induction step, $\cW_{\mu}$ defines an exact and faithful functor from $\cM^{\cO_{\mu}}(G_{n-\lambda_1})$ to the category of finite-dimensional vector spaces, and $\cW_{\mu}$ maps monomial representations and irreducible unitarizable representations  to one-dimensional spaces. Thus, so does $\cW_{\lambda}=\cW_{\mu}\circ E^{\lambda_1}$.
}
\proofend

\section{Choice-free definitions}\label{sec:NewDef}

\subsection{Generalized Whittaker models}\label{subsec:GenGenWhit}
In this section we define the generalized Whittaker model corresponding to a nilpotent element $e\in \fg$, without choosing   a neutral element $h$.
First of all, the filtration $\g_{\geq k}$ (unlike the grading $\g_k$) can be defined without choosing  $h$. It is in fact called the Deligne filtration and by \cite[\S I.6]{Del} is uniquely defined by the properties:
\begin{multline}
ad(e)(\g_{\geq k})\subset \g_{\geq k+2} \text{ and }\\
\text{the map }\g_{\geq -k}/\g_{\geq -k+1}\to \g_{\geq k}/\g_{\geq k+1} \text{ given by } ad(e)^k \text{ is an isomorphism.}
\end{multline}

It is easy to see that this filtration can be defined explicitly by
\begin{equation}
\g_{\geq k}:=\sum_{i\geq\max(1-k,1)} (\Ker(ad(e)^i)\cap \Im(ad(e)^{i+k-1}))
\end{equation}
We will sometimes denote this also by $\fg_{> k-1}$ or by  $\fg_{e,\geq k}$.

Let $e^\bot$ denote the orthogonal complement to $\{e\}$ under the Killing form $\langle \cdot , \cdot \rangle$ and let
\begin{equation}
\fu:=\fg_{\geq 1}, \quad \fv:=\fg_{> 1}, \quad I:=ad(e)^2(e^\bot)\cap \fv
\end{equation}

\begin{lemma}\label{lem:GenGenWhit}
\begin{enumerate}
\item \label{it:Iu} $I$ is an ideal in $\fu.$
\item \label{it:Iv} $\dim \fv/I=1$ and $e\notin I.$
\item \label{it:Iw} There exists a unique symplectic form $\omega$ on $\fu/\fv$ such that for any $a,b \in \fu$ we have $[a,b]-\omega(\bar a , \bar b)e\in I$, where $\bar a$ and $\bar b$  denote the classes of $a$ and $b$.
\item \label{it:IH}$\fu/I$ is a Heisenberg Lie algebra, and its center is spanned by class $\bar e$  of $e$.
\end{enumerate}
\end{lemma}
\begin{proof}
Note that $ad(e)^2(\fg_{\geq-2})=\fg_{\geq2}=\fv$ and $ad(e)^2(\fg_{>-2})=\fg_{>2}$. Pick an $\sl_2$-triple $(f,h,e)$. Using the $h$-grading it is easy to see that $\fg_{\geq k}$ is a Lie algebra filtration. Now\\
\eqref{it:Iu}
 $[\fu,I]\subset[\fu,\fv]=\fg_{>2}= ad(e)^2(\fg_{>-2})\subset I$. \\
\eqref{it:Iv} Since $e^\bot\cap \fg_{\geq -2}$ has codimension at most 1 in $\fg_{\geq -2}$, $I$ has codimension at most 1 in $\fv$. Thus it is enough to show that $e\notin I$, {\it i.e.} if $e=[e,[e,c]]$ then $\langle e,c\rangle\neq 0$. Since $ad(h)$ has integer eigenvalues, $\langle h,h\rangle \neq0$.
Now $[e,[e,-f/2]]=[e,-h/2]=e$ and $\langle e,-f/2\rangle =\langle[e,-h/2],-f/2\rangle=\langle-h/2,[e,f/2]\rangle=-1/4\langle h,h\rangle \neq0$.
Now, $c+f/2\in \ker(ad(e)^2)\subset \fg_{\geq -1}\subset e^\bot$, thus $\langle e,c\rangle=\langle e,-f/2\rangle\neq0$.\\
For \eqref{it:Iw}  let $c:=\langle f, e \rangle^{-1}f$ and define $\omega(\bar a, \bar b):=\langle c,[a,b]\rangle$. It is easy to see that $\omega$ is the only anti-symmetric form satisfying $[a,b]-\omega(\bar a , \bar b)e\in I$. Let us show that $\omega$ is non-degenerate. Let $a\in \fu$ such that $[a,b]\in I$ for any $b\in \fu$. This implies $[c,a]\in \fu^{\bot}=\fg_{>-1}$ and thus $a\in \fv$.\\
\eqref{it:IH} follows immediately from \eqref{it:Iw}.
\end{proof}


\begin{defn}
We now define a  character of the center of $\fu/I$ by requiring it to be 1 on $\bar e$, consider the corresponding oscillator representation of the Lie group $\Exp(\fu/I)$ and lift it to an irreducible representation $\Sc_e$ of $U:=\Exp(\fu)$. We then define the generalized Whittaker model associated to $e$ by $\cM_e:=\ind_U^G(\Sc_e)$.
\end{defn}

The connection to the generalized Whittaker models $\cW_f$ is given by the following straightforward lemma.

\begin{lem}\label{lem:GenGen}
Let $(f,h,e)$ be an $\sl_2$-triple and define $\varphi \in \fg^*$ by $\varphi(x):=\langle f, e \rangle^{-1} \langle f,x\rangle$. Then $\cM_e$ is naturally isomorphic to $\cW_{\varphi}$.
\end{lem}

\begin{remark}
The analogous approach in positive characteristic immediately faces two problems: exponentials not being defined and the Killing form being degenerate. However, for $\fg=\fg_n$ we can replace the Killing form by the trace form, and try to replace the exponential by the map $X \mapsto \Id +X$. Then the next question is whether Lemma \ref{lem:GenGenWhit}  holds. One can show that in three cases it fails completely: if $\charc \F=2, \, n\geq 3$, or $\charc \F =3, n\geq 8$, or $n \geq \charc \F -1 > 3$. In these cases there exists $e\in \fg_n$ such that $e\in \fg_{\geq 3}$ and $I=\fv$. We also see in these cases that the Deligne filtration is not a Lie algebra filtration.

In other cases we have, for any $e\in\fg$,   $e\in \fg_{\geq 2}$ but $e\notin \fg_{\geq 3}$ . It is not clear whether $\fu$ and $\fv$ are always Lie subalgebras or whether $I$ is an ideal in $\fu$. However, any $e\in \fg_n$ can be completed to an $\sl_2$-triple. Using this triple, one can show that if $\charc \F>2$ and $e^{(\charc \F +1)/2}=0$ then Lemma \ref{lem:GenGenWhit} holds, $\Id+\fu$ forms a subgroup of $G$ which includes $\Id+I$ as a normal subgroup, $e$ defines a central character of the Heisenberg group $(\Id+\fu)/(\Id+I)$ and one can consider the corresponding oscillator representation and Whittaker model.
\end{remark}

\subsection{Degenerate Whittaker models}


Let $Z$ be  a
 rational semi-simple element that commutes with $e$. For  any $t\in \Q$ define
\begin{equation}
\fg_{\geq t}^{e,Z}:=\sum_i \left (\fg_{\geq i}  \cap \sum_{s\geq t-i}\fg_s^Z  \right ).
\end{equation}
Note the following straightforward lemma.
\begin{lemma}
If $h\in \fg_Z$ is a neutral element for  $e$ then $\fg_{\geq t}^{e,Z}=\fg^{h+Z}_{\geq t}$
\end{lemma}

In fact, commuting $e$ and $Z$ is the same amount of information as the Lie algebra element $X=e+Z$. We can reformulate the filtration in terms of $X$. First define
\begin{equation}
\fg_{>k}^{X,t}:= \sum_{i\geq\max(-k,0)} (\Ker((ad(X)-t\Id)^i)\cap \Im((ad(X)-t\Id)^{i+k})).
\end{equation}
The following lemma is straightforward.
\begin{lemma}
$$\fg_{>k}^{X,t}=\fg_t^Z \cap \fg^e_{>k} \quad \text{and} \quad \fg_{>t}^{e,Z} = \sum_i\sum_{t\geq s-i} \fg_{>i}^{X,t}$$
\end{lemma}

Now define

\begin{equation}
\fu:= \fg_{\geq 1}^{e,Z} \quad \fv:= \fg_{>1}^{e,Z} \quad I:= ad(X)^2(e^\bot)\cap \fv \quad J:=I + ad(X)(\fu).
\end{equation}

\begin{lemma}\label{lem:ZE}
 $J$ is an ideal in $\fu$, and the algebra $\fu/J$ is isomorphic to the Heisenberg algebra defined in \S\ref{subsec:GenGenWhit} using the  element $e$ of the Lie algebra $\fg_Z.$
\end{lemma}
\begin{proof}
Note that for any $t\neq 0, \, \fg_{t}^Z\subset X^\bot$ and $ad(X)$ is invertible on $\fg_{t}^Z$. Thus 
$J \cap \fg_{t}^Z=\fu\cap \fg_{s}^Z$.

To see that $J$ is an ideal in $\fu$, let $a\in \fu\cap \fg_{s}^Z$ and $b\in J\cap \fg_{t}^Z$. Then $[a,b]\in \fu\cap \fg_{s+t}^Z$, which lies in $J$ unless $s+t=0$. If $s+t=0$ then $[a,b]\in (\fg^Z)_{>2}$, which lies in the ideal $I'$ defined by $e$ in  $(\fg^Z)_{\geq 2}$.

We also see that $\fu/J=(\fg^Z)_{\geq 2}/I'$.
\end{proof}

\begin{defn}
Using the isomorphism in Lemma \ref{lem:ZE} we define an oscillator representation of the Lie group $\Exp(u/I)$ and lift it to an irreducible representation $\sigma_e$ of $U:=\Exp(\fu)$. We then define the degenerate Whittaker model associated to $e$ and $Z$ by $\cM_{Z,e}:=\ind_U^G(\sigma_e)$.
\end{defn}

From the Lemmas \ref{lem:GenGen} and \ref{lem:ZE} we obtain

\begin{cor}
Let $Z\in \fg$ be a rational semi-simple element and let $(f,h,e)\in \fg_Z$ be an $\sl_2$-triple. Let $\varphi\in \fg^*$ be defined by pairing with $\langle f,e \rangle^{-1}f$. Then $\cM_{Z,e}$ is naturally isomorphic to $\cW_{h+Z,\varphi}$.
\end{cor}
\section{Global setting}\label{sec:Glob}
\setcounter{lemma}{0}

Let $K$ be a number field and let $\A=\A_{K}$ be its ring of adeles. In this section we let $\chi$ be a non-trivial unitary character  of $\A$, which is trivial on $K$. Then $\chi$ defines an isomorphism between $\A$ and $\hat{\A}$ via the map $a\mapsto \chi_{a}$, where $\chi_{a}(b)=\chi(ab)$ for all $b\in \A$. This isomorphism restricts to an  isomorphism
\begin{equation}\label{eq:chi_isomorphism}
 \widehat{\A/K}\cong \{\psi\in \hat{\A}\, | \psi|_{K}\equiv 1\}=\{\chi_{a}\, | \, a\in K\}\cong K.
\end{equation}
Given an algebraic group $G$ defined over $K$ we will denote its Lie algebra by  $\fg$ and we will denote the group of its adelic (resp. $K$-rational) points by $G(\A)$ (resp. $G(K)$). We will also define the Lie algebras $\fg(\A)$ and $\fg(K)$ in a similar way.

Given a Whittaker pair $(S,\varphi)$ on $\g(K)$, we set $\fu=\fg_{\geq 1}^{S}$ and $\fn$ to be the radical of the form $\omega_{\varphi}|_{\fu}$, where $\omega_{\varphi}(X,Y)=\varphi([X,Y])$, as before. Let $\fl\subset \fu$ be any choice of a maximal isotropic Lie algebra with respect to this form, and let $U=\exp \fu$, $N=\exp \fn$ and $L=\exp \fl$. Observe that we can extend $\varphi$ to a linear functional on $\g(\A)$ by linearity and, furthermore, the character $\chi_{\varphi}^{L}(\exp X)=\chi(\varphi(X))$ defined on $L(\A)$ is automorphic, that is, it is trivial on $L(K)$. We will denote its restriction to $N(\A)$ simply by $\chi_{\varphi}$. 

\begin{definition}
 Let $(S,\varphi)$ be a Whittaker pair for  $\g(K)$ and let $U,L, N,\chi_{\varphi}$ and $\chi_{\varphi}^{L}$ be as above. For an automorphic function $f$, we define its \emph{$(S,\varphi)$-Whittaker-Fourier coefficient} to be
\begin{equation}\label{eq:Whittaker-Fourier-coefficient}
 \cW\cF_{S,\varphi}(f):=\int_{N(\A)/N(K)}\chi_{\varphi}(n)^{-1}f(n)dn.
\end{equation}
We also define its \emph{$(S,\varphi,L)$-Whittaker-Fourier coefficient} to be
 \begin{equation}\label{eq:extended_Whittaker-Fourier_coefficient}
\cW\cF_{S,\varphi}^{L}(f):=\int_{L(\A)/L(K)}\chi_{\varphi}^{L}(l)^{-1}f(l)dl.
 \end{equation}
Observe that $\cW\cF_{S,\varphi}$ and $\cW\cF_{S,\varphi}^{L}$ define linear functionals on the space of automorphic forms. If $(\pi,V_{\pi})$ is an automorphic representation of $G$, then we will denote their restrictions to $\pi$ by $\cW\cF_{S,\varphi}(\pi)$ and $\cW\cF_{S,\varphi}^{L}(\pi)$ respectively.
\end{definition}

In order to adapt our arguments to the global setting we will have to replace Lemma \ref{lem:main} by the following one, which is analogous to \cite[Propositions 7.2 and 7.3]{GRS_Book}.

\begin{lemma}\label{lem:Glob}
 Let $(\pi,V_{\pi})$ be an automorphic representation of $G$. Then $\cW\cF_{S,\varphi}(\pi)\neq 0$ if and only if $\cW\cF_{S,\varphi}^{L}(\pi)\neq 0$. More specifically, if $\cW\cF_{S,\varphi}(f)\neq 0$ for some $f\in \pi$ then $\cW\cF_{S,\varphi}^{L}(\pi(u)f)\neq 0$ for some $u\in U(K)$.
\end{lemma}

\begin{proof}
\Dima{We assume that $\varphi$ is non-zero since otherwise the statement is a tautology.}
Let $f\in \pi$ be such that $\cW\cF_{S,\varphi}(f)\neq 0.$
Define a function $f_{\chi_{\varphi}^{L}}$ on $L$ by
\[
 f_{\chi_{\varphi}^{L}}(l)=\cW\cF_{S,\varphi}(\pi(l)f)
\]
and observe that the function $(\chi_{\varphi}^{L})^{-1}\cdot f_{\chi_{\varphi}^{L}}$ is left-invariant under the action of $N(\A)L(K)$. In other words, we can identify $(\chi_{\varphi}^{L})^{-1}\cdot f_{\chi_{\varphi}^{L}}$ with a function on
\begin{equation}\label{=LN}
L(\A)/  N(\A)L(K) \cong (L/N)(\A)\big/ (L/N)(K),
\end{equation}
where the  equality follows from the fact that $L/N$ is abelian.  Therefore, we have a Fourier series expansion
\begin{equation}
 f_{\chi_{\varphi}^{L}}(l)=\sum_{\psi\in (L(\A)/  N(\A)L(K) )^{\wedge}}c_{\psi,\chi_{\varphi}^{L}}(f)\psi(l)\chi_{\varphi}^{L}(l),
\end{equation}
where
\begin{equation}\label{eq:c_Whittaker-Fourier_coefficient}
 c_{\psi,\chi_{\varphi}^{L}}(f)=\int_{L(\A)/L(K)}\psi(l)^{-1}\chi_{\varphi}^{L}(l)^{-1}f(l)dl.
 \end{equation}
Since
\begin{equation}
 0\neq \cW\cF^L_{S,\varphi}(f)=f_{\chi_{\varphi}^{L}}(e)=\sum_{\psi\in (L(\A)/  N(\A)L(K) )^{\wedge}}c_{\psi,\chi_{\varphi}^{L}}(f),
\end{equation}
we conclude that at least one of the coefficients $c_{\psi,\chi_{\varphi}^{L}}(f)$ is different from 0.

Now observe that the map $X \mapsto \omega_{\varphi}(X,\cdot)=\varphi\circ \ad(X)$ induces an isomorphism between $\fu/\fl$ and $(\fl/\fn)'$. Hence, according to equations (\ref{eq:chi_isomorphism}) and \eqref{=LN}, we can use the character $\chi$ to define a group isomorphism
\[
 \begin{array}{rcl}
  (U/L)(K) & \longrightarrow & (L(\A)/  N(\A)L(K) )^{\wedge} \\
      u & \mapsto & \psi_{u},
 \end{array}
\]
where
\[
 \psi_{u}(l)=\chi(\varphi([X,Y])), \qquad  \mbox{$u=\exp X$}\quad \mbox{and}\quad \mbox{$l=\exp Y$.}
\]
Hence, for all $u\in U(k)$ and $l\in L$ we have
\begin{eqnarray*}
 \psi_{u}(l)\chi_{\varphi}^{L}(l)  =  \chi(\varphi([X,Y]))\chi(\varphi(Y))                   =  \chi(\varphi(Y+[X,Y]))\\
                   =  \chi(\varphi(e^{\ad(X)}(Y)))
                   =  \chi(\varphi(\Ad(u)Y))
                  =  \chi_{\varphi}^{L}(ulu^{-1}).
\end{eqnarray*}
Here we are taking again $u=\exp X$, $l=\exp Y$ and the middle equality follows from the vanishing of $\varphi$  on $\fg^{S}_{>2}$. But now, from formula (\ref{eq:c_Whittaker-Fourier_coefficient}) and the fact that $f$ is automorphic, we have
\begin{eqnarray*}
  c_{\psi_{u},\chi_{\varphi}^{L}}(f) & = & \int_{L(\A)/L(K)}\psi_{u}(l)^{-1}\chi_{\varphi}^{L}(l)^{-1}f(l)dl.
              =  \int_{L(\A)/L(K)}\chi_{\varphi}^{L}(ulu^{-1})^{-1}f(l)dl.\\
             & = & \int_{L(\A)/L(K)}\chi_{\varphi}^{L}(l)^{-1}f(u^{-1}lu)dl.               =  \cW\cF_{S,\varphi}^{L}(\pi(u)f),
\end{eqnarray*}
for all $u\in U(k)$. Since we have already seen that at least one of these coefficients is nonzero, we obtain the result claimed in the lemma.
\end{proof}

The rest of the proof  of Theorem \ref{thm:main} can be applied in the adelic setting, with the appropriate modifications. \DimaD{This implies Theorem \ref{thm:Glob}}.

\appendix

\section{Schwartz induction and the proof of Lemma \ref{lem:Frob}}\label{app:Frob}
\setcounter{lemma}{0}
%
%
%
%

We start with the following lemmas from functional analysis.

\begin{lemma}[{\cite[Theorem 50.1 and Proposition 50.1]{Tre}}]\label{lem:N}
Let $V$ and $W$ be Hausdorff locally convex complete topological vector spaces. Suppose that $V$ is a nuclear space. Then the projective and the injective topologies on $V \otimes W$ agree, and we will denote the completion with respect to these topologies by $V \hot W$. Moreover,
\begin{enumerate}[(i)]
\item $V^*$ is nuclear.
\item If $W$ is nuclear as well then $V\hot W$ is nuclear.
\item If $U\subset V$ is a closed subspace, then both $U$ and $V/U$ are nuclear.
\end{enumerate}
Here, $V^*$ and $W^*$ denote the strong dual.
\end{lemma}

\begin{lemma}[{\cite[formulas (50.18) and (50.19)]{Tre}}]\label{lem:NF}
Let $V$ and $W$ be \Fre spaces. Suppose that $V$ is a nuclear space. Then \begin{enumerate}[(i)]
\item $(V\hot W)^* \cong V^* \hot W^*$
\item $L(V,W)\cong V^* \hot W.$
\end{enumerate}
Here, $L(V,W)$ denotes the space of all continuous linear maps form $V$ to $W$, endowed with the compact-open topology.
\end{lemma}

\begin{lemma}[{\cite[Corollary 1.2.5 and Proposition 1.2.6]{dCl}}]\label{lem:SG}
$\,$
\begin{enumerate}[(i)]
\item \label{it:fun} Let $G$ be an affine real algebraic group. Then $\Sc(G)$ is a nuclear \Fre space and for any smooth \Fre representation $\pi$ of $G$ of moderate growth, we have a natural isomorphism
 $$\Sc(G) \hot \pi \cong \Sc(G,\pi).$$
\item For two smooth affine semi-algebraic varieties $M,N$ we have $\Sc(M)\hot \Sc(N)\cong \Sc(M\times N)$.
\end{enumerate}
\end{lemma}

\begin{lemma}\label{lem:IsActions}
 Given $f\in \Sc(G;\pi)$ and $x$, $g\in G$, we define
\begin{eqnarray*}
(R(g)f)(x) =  f(xg)       \quad  (R\pi(g)f)(x) = \pi(g)f(xg) \\
(L(g)f)(x) =  f(g^{-1}x)  \quad  (L\pi(g)f)(x) =  \pi(g)f(g^{-1}x).
\end{eqnarray*}
Then, all of the induced $G$-module structures on $\Sc(G;\pi)$ are isomorphic.
\end{lemma}

\begin{proof}
Given $f\in \Sc(G;\pi),$ let
\[
\widetilde{f}(x)=f(x^{-1})\qquad \mbox{for all $x\in G.$}
\]
Then
\begin{eqnarray*}
(\widetilde{R(g)f})(x)  =  (R(g)f)(x^{-1})
 =  f(x^{-1}g)
 =  \widetilde{f}(g^{-1}x)
 =  (L(g)\widetilde{f})(x).
\end{eqnarray*}
It's clear, then, that the map $f\mapsto \widetilde{f}$ defines a $G$-intertwining isomorphism between $(R,\Sc(G;\pi))$ and $(L,\Sc(G;\pi))$. Similarly, given $f\in \Sc(G; \pi)$, we set
\[
\widehat{f}(x)=\pi(x)f(x) \qquad \mbox{for all $x\in G.$}
\]
Since $\tau$ is of moderate growth, $\widehat{f}\in \Sc(G; \pi)$ and
\begin{eqnarray*}
\widehat{R\pi(g)f}(x)& = & \pi(x)(R\pi(g)f)(x) =  \pi(x)\pi(g)f(xg)  =  \widehat{f}(xg)=(R(g)\widehat{f})(x),
\end{eqnarray*}
that is, the map $f\mapsto \widehat{f}$ defines a $G$-intertwining isomorphism between the spaces $(R\pi,\Sc(G;\pi))$ and $(R,\Sc(G;\pi))$. The other isomorphisms are similar.
\end{proof}

\begin{corollary}\label{corollary:covariants}
 Under any of the above $G$-module structures,
\[
 \Sc(G; \pi)_{G}\cong  \pi,
\]
where $\Sc(G; \pi)_{G}$ is the space of $G$-coinvariants, i.e. the quotient of $\Sc(G; \pi)$ by the joint kernel of all $G$-invariant functionals.
\end{corollary}

%

\begin{lemma}\label{lemma:H_covariants}
We have
\[
 \ind_H^G(\rho)\cong (\Sc(G,\rho)\otimes \Delta_H)_{H}.
\]
\end{lemma}

\begin{proof}
Let $\Phi:\Sc(G,\rho) \to  \ind_H^G(\rho)$ denote the surjection that defines $\ind_H^G(\rho)$, see Definition \ref{def:ind}.
Observe that for all $\tilde{h}\in H$,
\begin{eqnarray*}
 (\Phi(R\rho(\tilde{h})f))(g)  =  \int_{H}\rho(h)(R\rho(\tilde{h})f(gh)dh
     =  \int_{H}\rho(h)^{}\rho(\tilde{h})f(gh\tilde{h})\, dh
 =  \Delta_{H}^{-1}(\tilde{h})\Phi(f)(g).
\end{eqnarray*}
Hence, the map $\Phi$ defines an $H$-invariant operator between $(R\rho\otimes  \Delta_{H},\Sc(G;\rho))$ and $\ind_H^G(\rho)$. Since this operator is surjective, it should factor through a surjective map
\begin{equation}\label{=map}
  (\Sc(G;\rho)\otimes \Delta_{H})_{H}\twoheadrightarrow \ind_H^G(\rho).
\end{equation}
We want to show that this map is injective. Let us fix $T\in [(\Sc(G;\rho)\otimes\Delta_{H})^{*}]^{H}$. Now according to \cite[Lemma 2.2.5]{dCl}, there exists a semi-algebraic open cover $\{U_{k}\}_{k=1}^{n}$ of $G/H$ and a tempered partition of unity $\{\gamma_{k}\}_{k=1}^{n}$ subordinated to $\{U_{k}\}_{k=1}^{n}$ such that
\begin{equation}\label{eq:good_open_cover}
 p^{-1}(U_{k})\cong  U_{k}\times H,
\end{equation}
where $p:G\longrightarrow G/H$ is the natural projection. Using the above equation, we can define a partition of unity $\{\eta_{k}\}_{k=1}^{n}$ subordinated to the open cover $\{p^{-1}(U_{k})\}$ consisting of $H$-invariant functions. Furthermore, given $f\in \Sc(G;\rho),$
\[
 \Phi(\eta_{k}f)=\gamma_{k}\Phi(f), \qquad \mbox{$k=1,\ldots,n$.}
\]
Therefore, we have a decomposition
\[
 T=\sum_{k=1}^{n}\eta_{k}T
\]
and we can identify each $\eta_{k}T$ with an element of $[(\Sc(p^{-1}(U_{k});\rho)\otimes \delta_{H}^{-1})']^{H}$. According to  (\ref{eq:good_open_cover})
\[
 \Sc(p^{-1}(U_{k}))\cong \Sc(H)\, \widehat{\otimes}\, \Sc(U_{k})
\]
and hence, according to Corollary \ref{corollary:covariants} $[(\Sc(p^{-1}(U_{k});\rho)\otimes \Delta_{H})^{*}]^{H}\cong \Sc(U_{k};\rho)^{*},$ that is, for all $k=1,\ldots,n,$ there exists $\tilde{T}_{k}\in \Sc(U_{k};\rho)^{*}$ such that
\begin{eqnarray*}
 \eta_{k}T(f) & = & \tilde{T}_{k}(\Phi(\eta_{k}f))                     =   \tilde{T}_{k}(\gamma_{k}\Phi(f))
          =  (\gamma_{k}T_{k}\circ \Phi)(f).
\end{eqnarray*}
But we can see $\gamma_{k}T_{k}$ as an element of $\ind_H^G(\rho)'$. Therefore, if we set $\tilde{T}=\sum_{k} \gamma_{k}\tilde{T}_{k}$, then
\[
 T=\sum_{k} \eta_{k}T=\sum_{k} \gamma_{k}\tilde{T}_{k}\circ \Phi=\tilde{T}\circ \Phi.
\]
We have thus shown that any element of $[(\Sc(G;\rho)\otimes\Delta_{H})^*]^{H}$ factors through $\ind_H^G(\rho)^{*}$ which proves the injectivity of the map \eqref{=map}.
\end{proof}

\begin{proof}[Proof of Lemma \ref{lem:Frob}]
We have
\begin{multline*}
\Hom_{G}(\ind_H^G(\rho),\pi^*) \cong ((\DimaA{(\Sc(G,\rho)\otimes \Delta_H)}^*)^H\hot \pi^*)^G\cong
  \Hom_{G\times H}(\Sc(G,\rho)\otimes \Delta_H,\pi^*)\cong\\
  \Hom_{G\times H}(\Sc(G,\rho)\hot \pi\otimes \Delta_H,\C) \cong \Hom_{G\times H}(\rho \hot \Sc(G,\pi)\otimes \Delta_H,\C)
\end{multline*}
Here, $G$ acts on $\Sc(G,\pi)$ by $L\pi$, while $H$ acts diagonally: on $\rho\otimes \Delta_H$ and on $\Sc(G,\pi)$ by $R$. This action is isomorphic to an action in which $G$ acts on $\Sc(G,\pi)$ by $L$ and $H$ by $R\pi$. Under this action we have $$ \Hom_{G\times H}(\rho \hot \Sc(G,\pi)\DimaA{\otimes \Delta_H},\C)\cong ((\Sc^*(G))^G\hot \pi^*\hot \rho^*\otimes \Delta_H^{-1})^H.$$
Since all left $G$-invariant distributions on $G$ are proportional and right $\Delta_G$-equivariant, the latter space is isomorphic to $\Hom_H(\rho,\pi^*\Delta_H^{-1}\Delta_G)$.
\end{proof}

\end{document}